\documentclass[a4paper,11 pt]{amsart}

\usepackage[sc]{mathpazo}
\usepackage{eulervm}

\usepackage[utf8]{inputenc}

\usepackage[french]{}
\usepackage{newlfont}
\usepackage{amsmath, amscd}
\usepackage{amsthm}
\usepackage{amssymb}
\usepackage[a4paper,top=3.5cm,bottom=3.5cm,left=3.5cm,right=3.5cm,heightrounded]{geometry}
\usepackage[font=small,labelfont={bf}]{caption}
\usepackage{subfig}
\usepackage{caption}
\usepackage{graphicx}
\usepackage[T1]{fontenc}
\usepackage{listings}
\usepackage{setspace}
\usepackage{mathtools}
\usepackage{textcomp}
\usepackage{array}
\usepackage{multicol}
\usepackage[bbgreekl]{mathbbol}
\usepackage{bm}
\usepackage{braket}

\usepackage{chngcntr}
\counterwithin{equation}{section}

\newcommand{\N}{\mathcal{N}}
\newcommand{\prt}[1]{\frac{\partial}{\partial{#1}}}

\newcommand{\eqeu}{\mathrm{e}_\T}

\newcommand{\rk}{\mathrm{rk}}
\newcommand{\pt}{\mathrm{pt}}

\newcommand{\Z}{\mathbb{Z}}
\newcommand{\C}{\mathbb{C}}
\newcommand{\Q}{\mathbb{Q}}

\newcommand{\res}{\mathop{\text{Res}}}
\newcommand{\thres}{\mathop{\emph{Res}}}

\newcommand{\op}[1]{\mathop{\mathchoice{\mbox{\rm #1}}{\mbox{\rm #1}}
{\mbox{\rm \scriptsize #1}}{\mbox{\rm \tiny #1}}}\nolimits}

\newcommand{\M}{{\mathcal M}}
\newcommand{\GL}{{\mathrm G\mathrm L}}
\newcommand{\SL}{{\mathrm S \mathrm L}}
\newcommand{\PGL}{{\mathrm P \mathrm G \mathrm L}}
\newcommand{\Dol}{{\mathrm{Dol}}}
\newcommand{\B}{{\op{B}}}
\newcommand{\bE}{{\mathbb E}}

\DeclareMathOperator{\End}{End}
\DeclareMathOperator{\im}{im}
\newcommand{\T}{{{\mathbb T}}}
\newcommand{\calC}{{\mathcal C}}
\newcommand{\calN}{{\mathcal N}}
\newcommand{\calA}{{\mathcal A}}
\newcommand{\h}{{\mathrm h}}
\newcommand{\cM}{{\overline{\M}}}

\newcommand{\G}{{\mathrm G}}

\newcommand{\ch}{\operatorname{ch}}
\newcommand{\Sp}{\operatorname{Sp}}

\newcommand{\bes}{\begin{eqnarray*}}
	\newcommand{\ees}{\end{eqnarray*}}
\newcommand{\beq}{\begin{eqnarray}}
	\newcommand{\eeq}{\end{eqnarray}}

 \theoremstyle{plain}
 \newtheorem{theorem}{Theorem}[section]
\newtheorem{proposition}[theorem]{Proposition}
 \newtheorem{lemma}{Lemma}

 \newtheorem{conjecture}{Conjecture}

\newtheorem{corollary}[theorem]{Corollary}
\theoremstyle{definition}
 \newtheorem{definition}[theorem]{Definition}

 \newtheorem{remark}[theorem]{Remark}

\newcommand{\D}{\text{d}}

\title{An enumerative approach to $P=W$ }
\author{S. M. Chiarello, T. Hausel, A. Szenes}

\begin{document}
	
		\begin{abstract} The $P = W$ conjecture identifies the perverse filtration of the Hitchin system on the cohomology of the moduli space of Higgs bundles with the weight filtration of the corresponding character variety. In this paper, we introduce an enumerative approach to to this problem; our technique only uses the structure of the equivariant intersection numbers on the moduli space of Higgs bundles, and little information about the topology of the Hitchin map.  In the rank $2$ case, starting from the known intersection numbers of the moduli of stable bundles, we derive the equivariant intersection numbers on the Higgs moduli, and then verify the top perversity part of our enumerative $P = W$ statement for even tautological classes.  A key in this calculation is the existence of polynomial solutions to the Discrete Heat Equation satisfying particular vanishing properties. For odd classes, we derive a determinantal criterion for the enumerative $P = W$. 
	\end{abstract}
	
	\maketitle
	
	\section{introduction}

\subsection{Moduli spaces of Higgs bundles}\label{intro1}

Let $C$ be a smooth complex projective curve of genus $g\ge 2$; the canonical bundle of $C$ will be denoted by $K$.
A \textit{Higgs bundle} is a pair $(E,\Phi)$, where $E$ is a vector bundle on $C$, and $\Phi:E\to E\otimes K$ is a bundle map; the pair is \textit{stable} if for a proper subbundle $S\subset E$ satisfying $\Phi S\subset S\otimes K$, one has 
$\deg S/\rk S<\deg E/\rk E$.

As in \cite[1.2.2]{decataldo-hausel-migliorini}, we use the notation $\M_\Dol(\GL_2)$  for the moduli space of rank-$2$, degree-$1$ stable Higgs bundles\footnote{Note that all bundles in this paper will assumed to be of degree 1, and, for simplicity, the degree will be omitted from the notation.} on $C$. This is a smooth quasiprojective variety of dimension $8g-6$. 
 
We fix a degree-$1$ line bundle $\Lambda\in J_1(C)$, and denote by $\M_\Dol(\SL_2)$ the moduli space of rank-$2$ stable Higgs bundles with fixed determinant $\Lambda$, and trace-zero Higgs field:
   \[ \M_\Dol(\SL_2) = \set{(E,\Phi)\in \M_\Dol(\GL_2)|\,\det E\simeq \Lambda,\,\mathrm{Tr}\Phi=0}; \] 
this is a smooth quasiprojective variety of dimension $6g-6$. 
   
Finally, note that the finite group $\Gamma\coloneqq J_0(C)[2]$ of $2$-torsion points on the Jacobian acts by tensorization on $\M_\Dol(\SL_2)$: for $L\in\Gamma$	 and  $(E,\Phi)\in \M_\Dol(\SL_2)$, we let $L\cdot(E,\Phi)\mapsto(L\otimes E,\Phi)$. 
The quotient by the $\Gamma$-action is an orbifold 
$$\M_\Dol(\PGL_2)\coloneqq \M_\Dol(\SL_2)/\Gamma,$$
which is the the odd-degree component of the $\PGL_2$-Higgs moduli space. 
As this space will be the protagonist of our story, \textit{we will denote it simply by $\M$.}

There is a parallel set of moduli spaces of stable bundles on $C$ (without the Higgs field $\Phi$), which we will denote in a similar fashion, but replacing $\M_\Dol$ by $\N$. For example, $\N(\GL_2)$ is simply the (projective) moduli space of rank-2 stable bundles on $C$, whose dimension is $4g-3$,  which is exactly half of the dimension of $\M_\Dol(\GL_2)$. In fact, there is an embedding $T^*\N(\GL_2)\hookrightarrow\M_\Dol(\GL_2)$, under which the zero-section of $T^*\N(\GL_2)$ goes to the set of Higgs bundles with $\Phi=0$. A similar statement holds for the other two moduli spaces, as well. We will use the consistent notation $\N=\N(\PGL_2)$.

The relationship between the cohomologies of our Higgs moduli spaces is as follows: 
$$H^*(\M_\Dol(\GL_2))\cong H^*(\M_\Dol(\PGL_2))\otimes H^*(T^*J_0(C))$$ and $$H^*(\M_\Dol(\PGL_2))\cong H^*(\M_\Dol(\SL_2))^\Gamma.$$

Understanding these cohomology groups has been the focus of a large body of research. They have many applications, from their connection to mirror symmetry \cite{hausel-thaddeus3, groechenig-wyss-ziegler1}, to Ng\^{o}'s geometric fundamental lemma \cite{ngo,hausel-gths,groechenig-wyss-ziegler2} and knot invariants \cite{mellit}.

\subsection{Equivariant cohomology and generators}\label{sec:generators}

There is a rescaling $\C^\times$-action on all our Higgs moduli spaces:
\[ \lambda\cdot(E,\Phi)= (E,\lambda\Phi).\] For simplicity of notation, we will often use $\T\coloneqq \C^\times$. 

The equivariant cohomology $H_\T(\M)$ is a finitely generated module over the equivariant cohomology of a point $H_\T(\pt)$, which we will identify with the polynomial ring in a single variable $u$:
\[ H_\T(\pt) = H^*(B\T)\cong \C[u].  \]

All our Higgs moduli spaces are semi-projective with respect to the $\T$-action \cite{hausel-large}, and this, in particular, implies its formality: additively, we have a $H^*(B\T)\cong \C[u]$-module isomorphism \begin{eqnarray} \label{formal} H^*_\T(\M)\cong H^*(\M)\otimes H^*(B\T)= H^*(\M)[u].\end{eqnarray}

In \cite{hausel-thaddeus1},  a universal Higgs bundle endowed with a compatible $\T$-action over $\M\times C$ was constructed. While the rank-$2$ vector bundle $\bE$ is only unique up to tensoring with a line bundle on $\M$, the rank-$4$, $\T$-equivariant vector bundle  $\mathrm{End}(\bE)$ is unambiguously defined.    Now we fix an appropriate basis of $H^*(C)$: 
\begin{itemize}
	\item we denote by 1 the canonical generator of $H^0(C)$;
	\item we denote by $\omega$ the Poincar\'e dual of the class of a point in $H^2(C)$;
	\item finally, we choose elements $e_1,\dots,e_{2g}\in H^1(C)$, which form a symplectic basis of $H^1(C)$, i.e. for $i<j$, they satisfy $e_ie_j=\delta_{i+g-j,0}\cdot\omega$.
\end{itemize}

The K\"unneth decomposition of the second $\T$-equivariant Chern class of 
${\End }(\mathbb E)$ 
\begin{equation}
\label{defabps}
c_2({\End }({\mathbb E}))=2{\alpha} \otimes \omega + 4\sum_{i=1}^{2g} \psi_i \otimes e_i - \beta \otimes 1 \in H_\T^*(\M)\otimes H^*(C) 
\end{equation}
provides us\footnote{Note that the definition of the universal classes in \cite[(1.2.10)]{decataldo-hausel-migliorini} as well as in  \cite[5.1]{hausel-villegas} do not have the correct scalars. The correct ones are as in \eqref{defabps} and as in \cite[(1.5)]{hausel-thaddeus2}. This discrepancy in the scalars does not effect the arguments in \cite{decataldo-hausel-migliorini,hausel-villegas}.} with well-defined equivariant classes $\alpha\in H_\T^2(\M)$, $\psi_i\in H_\T^3(\M)$ and $\beta\in H_\T^4(\M)$. 
It is proved in \cite{hausel-thaddeus2} that $\alpha,\psi_i$ and $\beta$ generate the $\T$-equivariant cohomology ring $H^*_\T(\M)$ as an $H^*(B\T)$ algebra. Their images in ordinary cohomology, in other words, the K\"unneth components of the second non-equivariant Chern class of the vector bundle $\End(\bE)$, generate $H^*(\M)$. One can use this observation to give an explicit embedding $H^*(\M)\to H^*_\T(\M)$ yielding \eqref{formal}. For this reason, we will use the same notation $\alpha$,$\beta$ and $\psi_i$, for the K\"unneth components of the second non-equivariant Chern class of  $\End(\bE)$ as well.

\subsection{Character varieties, the nonabelian Hodge theorem and the weight filtration}

We define the {\em $\GL_2$-character variety} as the affine GIT quotient by the diagonal adjoint group action as follows:
\bes\set{ A_1,B_1,\dots, A_g,B_g \in \GL_2\ |\ 
A_1^{-1} B_1^{-1} A_1 B_1 \dots A_g^{-1} B_g^{-1} A_g B_g = - {\mathrm I} }/\! /\PGL_2.
\ees 
The result is a smooth affine variety of dimension $8g-6$ which we denote by $\M_\B(\GL_2)$ (the \textit{Betti moduli space}). We also define the {\em $\SL_2$-character variety}: 
\bes{{\M}}_\B(\SL_2) \coloneqq \set{ A_1,B_1,\dots, A_g,B_g \in {\SL}_2\ | \
A_1^{-1} B_1^{-1} A_1 B_1 \dots A_g^{-1} B_g^{-1} A_g B_g = - {\mathrm I} }/\!/{\mathrm PGL}_2, \ees which is again smooth, affine, and
 of dimension $6g-6$. 

Finally, denote by $ \bbmu_2\coloneqq \{\pm I\}$ the center of $\SL_2$. Then 
$\bbmu_2^{2g}$ acts on $\SL_2^{2g}$ by coordinate-wise multiplication. The quotient 
\begin{equation}\label{quotient}{\M}_\B(\PGL_2) \coloneqq {\M}_{\B}(\GL_2)/\!/ \GL_1^{2g}={\M}_\B(\SL_2)/\bbmu_2^{2g}\end{equation}
is the (odd component of the) {\em $\PGL_2$-character variety}. Again,
 $\M_\B(\PGL_2)$ is an affine orbifold of dimension $6g-6$. 

The cohomology of the varieties $H^*(\M_\B(\G))$ for all of our groups $G$ carries Deligne's weight filtration: $$W_0(H^*(\M_B(\G)))\subset W_1(H^*(\M_B(\G)))
\subset \dots \subset H^*(\M_B(\G))).$$

Recall from \cite[Definition 4.1.6]{hausel-villegas} (c.f. \cite[(1.2.6)]{decataldo-hausel-migliorini}) that we define 
a class $x\in H^i(\M)$ to have \textit{homogeneous weight} $k$ if $$x\in W_{2k}(H^i(\M,\C))\cap F^k(H^i(\M,\C)),$$ where $F$ denotes the Hodge filtration in the mixed Hodge structure of Deligne. This provides us with a new grading on $H^*(\M)$: indeed, the product of two classes of homogeneous weight $k_1$ and $k_2$ respectively will be a class of homogeneous weight $k_1+k_2$. 
It is shown in \cite[Proposition 4.1.8]{hausel-villegas} that the universal classes $\alpha,\psi_i,\beta$ all have homogeneous weight $2$.

 Let us  denote by $W^i_{k,k}(\M)\subset H^i(\M)$ the vector space of classes of homogeneous weight $k$ and degree $i$.   As $H^*(\M)$ is generated by universal classes, we have the decomposition $$H^i(\M)\cong \bigoplus_{k=1}^i W^i_{k,k}(\M),$$ which thus splits  the weight filtration (c.f. \cite[(1.2.5)]{decataldo-hausel-migliorini}) in the sense that \begin{equation}\label{defweight}W_{2k}(H^i(\M))\cong \bigoplus_{d=1}^k W^i_{d,d}(\M).\end{equation}

\subsection{The Hitchin map and the P=W conjecture}

The complex manifold  underlying the variety $\M$ was first constructed by Hitchin in \cite{hitchin} using gauge theory. Hitchin observed that the complex manifold $\M$ inherits a natural hyperk\"ahler metric from the gauge theory construction, and that in another complex structure of the hyperk\"ahler family that complex manifold is in fact the character variety we introduced above. This gives rise to the diffeomorphism \beq \label{nah} \M_\Dol(\G)\cong \M_\B(\G) \eeq for a reductive group $G$. These diffeomorphisms were reinterpreted by Simpson \cite{simpson} as the non-Abelian Hodge theorems.

Let us now recall the Hitchin map 
\begin{equation}\label{hitchinmap}
	\h \colon \M\to \calA \coloneqq H^0(C;K^2)
\end{equation}
defined by taking the determinant of the Higgs field $$\h(E,\Phi)=\det(\Phi)\in H^0(C;K^2).$$ It is a proper, completely integrable Hamiltonian system, in particular, the generic fibers are torsors for Abelian varieties. 

In what follows $\G$ will stand for one of $\GL_2,\SL_2$ or $\PGL_2$. As explained in \cite[\S 1,4]{decataldo-hausel-migliorini} the proper map $h$ induces the perverse filtration $$P_0(H^*(\M_\Dol(\G)))\subset P_1(H^*(\M_\Dol(\G)))
\subset \dots \subset H^*(\M_\Dol(\G))).$$

The main result of \cite{decataldo-hausel-migliorini} is
\begin{theorem}[$P=W$] \label{p=w} For $\G=\GL_2,\SL_2$ or $\PGL_2$ we have
	$$P_kH^*(\M_\Dol(\G))\cong W_{2k}(H^*(\M_B(\G))) \cong W_{2k+1}(H^*(\M_B(\G)))$$ induced by the non-Abelian Hodge theory diffeomorphisms in \eqref{nah}. 
	\end{theorem}

The proof in \cite{decataldo-hausel-migliorini} of Theorem~\ref{p=w} was complex: besides using
the knowledge of the cohomology of $\M_\Dol(\G)$ from \cite{hausel-thaddeus1}, \cite{hausel-thaddeus2} and \cite{hausel-compact} and the structure of the weight filtration on $H^*(\M_\B(\G))$ from \cite{hausel-villegas},
it also used a detailed description of the cohomology of singular Hitchin fibers. 

In this paper, we introduce a technique, which should yield an alternative proof of Theorem~\ref{p=w}, and only uses the information of equivariant intersection numbers of $\M_\Dol(\G)$. 

For higher rank Higgs bundles, recently, \cite{decataldo-maulik-shen} proved  $P=W$ 
 for genus $2$. Our approach, in principle, offers an alternative enumerative attack on the more general cases, but the computational aspects remain difficult. 
 
\subsection{Contents of the paper}

In Section $2$, we give the statement of the Enumerative P=W Theorem \ref{enumerativepw}, proving that it is equivalent to the classical one: Theorem \ref{p=w}. This theorem relies on the existence of a $\T$-equivariant compactification $i \colon \M\rightarrow \cM$ such that $$i^*P_kH^*(\cM) = P_kH^*(\M) $$ and on the explicit description of the perverse filtration on $\cM$ in \cite{decataldo-migliorini}.

In Section $3$ we compute the equivariant integrals on $\M$ in the sense of \cite{hausel-proudfoot}, obtaining a residue formula in Theorem \ref{equintm}. We also prove that the natural $\Sp(2g,\mathbb{Z})$-action on $H^*_\T(\M)$ preserves the perverse filtration. We then compute the intersection numbers on the infinity divisor $Z\subseteq \cM$, and provide a particularly simple residue formula for them in Proposition \ref{otherres}. We introduce the notion of defect of a class in $H^*(Z)$, closely related to the perversity, and we show in particular, that integrating a class of top defect amounts to computing the residue of a form with a simple pole. 

In Section $4$, we give a formula for the least-defect part of the lifts of the classes $\beta^k$ in Theorem \ref{topperv}, and we show that it is unique in Theorem \ref{uniqueness}. Moreover, in Proposition \ref{pervgrading}, we recover  the known fact that for the classes $\beta^k$, the perverse filtration is actually a grading, i.e. $\beta^k\notin P_{2k-1}H^{4k}(\M)$.

in Theorem \ref{ifinvertible}, we show  that for general classes of type $\beta^{k-h}\gamma^h$,  computing their perversities amounts to proving the non-vanishing of a particular polynomial determinant; we finally prove this nonvanishing in the case $h=1$ in Proposition \ref{h=1}.

 \paragraph{\bf Acknowledgements.} We would like to thank Mark de Cataldo, Camilla Felisetti and Luca Migliorini for useful discussions. The work of T.H. was partially supported by the Advanced Grant “Arithmetic and physics of Higgs moduli spaces” no. 320593 of the European Research Council. S.Ch. and A.Sz. gratefully acknowledge the support of the  Swiss National Science Foundation grants 17599 and 156645, as well as the NCCR SwissMAP.
  
 \section{The enumerative P=W conjecture}

\subsection{The compactification of $\M_\Dol(\G)$}

Recall the Bialynicki-Birula partition of the semiprojective variety $\M$. For $p \in \M^\T$ we define $$U_p\coloneqq \set{ x\in X | \lim_{\lambda\to 0}  \lambda x = p }$$ {\em upward flow} from $p$ and $$D_p\coloneqq \set{ x\in X | \lim_{\lambda\to \infty}  \lambda x = p}$$ {\em downward flow} from $p$.  Then $$\M=\cup_{p\in \M^\T} U_p$$ is the Bialinycki-Birula partition of $\M$ and $$\calC\coloneqq\cup_{p\in \M^\T} D_p$$ is the Bialinycki-Birula partition of $\calC\subset \M$, called the core of $\M$. Let $\M^\T=\cup F_{i}$ be the decomposition of the fixed point set into connected components; then 
$$U_i=\cup_{p\in F_i}U_p\subset \M$$
 and 
 $$D_i=\cup_{p\in F_i}D_p\subset \M$$ 
 are affine bundles over $F_i$.  

 A compactification $\cM\supset \M$ was constructed in \cite{hausel-compact}. The construction there was with symplectic cutting, producing a projective variety $\cM$. There is an algebraic version of this construction explained in \cite{hausel-large} for a general semiprojective variety. This yields the following in our case:  $\M$ comes with an ample line bundle $L\in Pic(\M)\cong \Z$ the generator of the Picard group. The $\T$-action can be linearized, and with an appropriate linearization we can construct the GIT quotient $$Z\coloneqq \M/\! /\T=(\M\setminus \cup_i D_i)/\T$$
 This is a projective orbifold of dimension $6g-7$. We can add it as divisor at infinity  to compactify $\M$ as follows. $$\cM\coloneqq (\M\times\C)/\!/ \T=(\M\times\C \setminus (\cup_i D_i)\times \{0\})/\T.$$ 
$\cM$ is a projective orbifold of dimension $6g-6$, which has the decomposition $\cM=\M\cup Z$. From the quotient construction, we 
have the Kirwan map \begin{equation}\label{Kirwanmap}
	\kappa_\cM:H^*_\T(\M)\cong H^*_\T(\M\times \C) \to H^*(\cM)
\end{equation}
 which is surjective \cite{kirwan}. Thus $H^*(\cM)$ is generated by $\kappa_\cM(\alpha), \kappa_\cM(\psi_i), \kappa_\cM(\beta)$ and $\kappa_\cM(u)$. By abuse of notation, we will denote these by $\alpha\in H^2(\cM))$, $\psi_i\in H^3(\cM))$, $\beta\in H^4(\cM))$, respectively, and the additional new class by $\eta=\kappa_\cM(u)\in H^2(\cM))$.

We also see that the Hitchin map \eqref{hitchinmap} is $\T$-equivariant when $\T$ acts on $\calA$ with weight 2, and thus we can extend the Hitchin map as follows \cite[\S]{hausel-compact}: $$\overline{\h}:\cM\to \overline{\calA}, \;\text{where }
\overline{\calA}\cong{\mathbb P}(\calA\times \C)$$  is a projective space of dimension $3g-3$. 

We have the \textit{perverse filtration} on $H^*(\cM)$
$$P_0(H^*(\cM))\subset P_1(H^*(\cM))
\subset \dots \subset P_{12g-12}(H^*(\cM)),$$ where $ P_p(H^*(\cM)) $ is defined by
$$\im\left(H^{*-\dim(\overline{\calA})}(\overline{\calA}, ^{\mathfrak p \! }\tau_{\leq p}( Rh_*\Q[\dim \overline{\calA}]))\to H^{*-\dim(\overline{\calA})}(\overline{\calA},Rh_*\Q[\dim \overline{\calA}]) \right).$$

Here  $^{\mathfrak p }\tau_{\leq p}$ is the perverse truncation functor (cf. \cite{decataldo-hausel-migliorini}).
\subsection{The enumerative version of P=W} Using the compactification $\cM$, we have another way to characterize the perverse filtration on $\M$.

\begin{proposition}\label{pwlift} Let $i:\M\to \cM$ be the natural embedding. Then $x\in P_k(H^j(\M))$ if and only if  there exists a $y \in P_{k-1}(H^{j-2}(\cM))$ such that $(x+\eta y) \eta^{3g-2-j+k}=0$.
\end{proposition}

\begin{proof} 
The perverse truncation functor $^{\mathfrak p }\tau_{\leq p}$ commutes with the  restriction to the open $\calA\subset \overline{\calA}$, and thus it follows that the embedding $i:\M\to \cM$ induces the inclusion
   \beq \label{contains}i^*(P_k(H^*(\cM))\subset P_k(H^*(\M)).\eeq

On the other hand, from the $(\eta,L)$-decomposition\footnote{Note that we use the notation $\eta=\eta_Z$ for the class $L$ in \cite{decataldo-migliorini} while $\eta$ in  \cite{decataldo-migliorini} denotes an ample class on $\cM$.} of
\cite[Corollary 2.1.7]{decataldo-migliorini}, we see that $i^*(P_k(H^*(\cM))$ induces a filtration on $i^*(H^*(\cM))\cong H^*(\M)$ satisfying a relative Hard Lefschetz theorem of the same type as the perverse filtration $P$ on $H^*(\M)$. It follows from this and  \eqref{contains}  that actually \beq \label{equal}i^*(P_k(H^*(\cM)))= P_k(H^*(\M)).\eeq 

On the projective $\cM$ we can use \cite[Proposition 5.2.4.(39)]{decataldo-migliorini} to deduce\footnote{Proposition 5.2.4 is claimed for a smooth total space, but as explained in \cite[Theorem 2.3.1]{decataldo-migliorini} the results hold for non-smooth varieties for intersection cohomology, and thus for orbifolds with ordinary cohomology.} that \beq  \label{decamiglio} P_k(H^j(\cM))=\sum_{l=0}^{3g-2-j+k} \im(\eta^l)\cap \ker(\eta^{3g-2-l-j+k}),\eeq where $\eta:H^{j}(\cM)\to H^{j+2}(\cM)$ denotes multiplication by $\eta\in H^2(\cM)$ by an abuse of notation.  
 
Now if $0\neq x\in H^j(\M)$ and $\widetilde{x}\in H^j(\cM)$ is a lift of $x$, then $\widetilde{x}\notin\im(\eta)$, since $i^*(\eta)=0$. It follows then from \eqref{equal} and \eqref{decamiglio} that 
  $$x\in P_k(H^j(\cM)) \Leftrightarrow \eta^{3g-2-j+k}\widetilde{x} = 0 \text{ for some } \widetilde{x}\in H^j(\cM) \text{ lift of } x.$$ 
  
Considering the composition $H^*(\M)\to H^*_\T(\M)$ (cf. \S\ref{sec:generators}) with the Kirwan map \eqref{Kirwanmap}, it is clear that 
 any lift of $x$ will have the form $\widetilde{x} = x+\eta y$ for some $y\in H^{j-2}(\cM)$. Finally, again by \eqref{equal} and \eqref{decamiglio}, we see that $x\in P_k(H^j(\M))$ if and only if such $y\in P_{k-1}(H^{j-2}(\cM))$, and this completes the proof.
\end{proof}
 
 After these preparations,  we are ready to formulate the our enumerative version of P=W.
 
 \begin{theorem}[Enumerative P=W] \label{enumerativepw} The P=W Theorem \ref{p=w} holds for $G=PGL_2$ if and only if for all $g\ge 2$, $l\geq 0$ and $m\geq 0$ there is an  extension $\beta^l\gamma^m+\eta q(\alpha,\beta,\gamma,\eta)$ of $\beta^l\gamma^m\in H^*(Z)$ such that \beq \label{intersectionz}\int_Z \eta^{3g-3-2l-2m}(\beta^l\gamma^m+\eta q(\alpha,\beta,\gamma,u))x=0\eeq for all $x\in H^*(Z)$. 
\end{theorem}
 
\begin{proof} 

Notice that for any class $v\in H^*(\cM)$, we have $\int_\cM v \eta = \int_Z v$, and the restriction map $H^*(\cM)\to H^*(Z)$ is surjective as both rings are generated by the universal classes and $\eta$. Therefore, Equation \eqref{intersectionz} is equivalent to 
$$\int_{\cM}  \eta^{3g-2-2l-2m}(\beta^l\gamma^m+\eta q(\alpha,\beta,\gamma,\eta))x=0$$ 
for all $x\in H^*(\cM)$. By Poincaré duality, this is, in turn, equivalent  to \begin{eqnarray} \label{vanishing} \eta^{3g-2-2l-2m}(\beta^l\gamma^m+\eta q(\alpha,\beta,\gamma,\eta))=0. \end{eqnarray}
By Proposition \ref{pwlift},  \eqref{vanishing} implies that $\alpha^k\beta^l\gamma^m\in P_{2k+2l+4m}(H^{2k+4l+6m}(\mathcal{M}))$ for all $k$, $l$ and $m$. Since $\alpha$, $\beta$ and $\gamma$ have weight respectively $2$, $2$ and $4$, Proposition \ref{rmksigma} and \eqref{defweight} imply that $$P_kH^*(\mathcal{M}) \subseteq W_{2k}(\mathcal{M}) $$ for all $k\ge 0$. Finally, as in \cite{decataldo-hausel-migliorini}, this, the relative Hard Lefschetz theorem \cite[Theorem 2.1.4]{decataldo-migliorini} and the curious Hard Lefschetz theorem \cite[Theorem 1.1.5]{hausel-villegas} imply Theorem~\ref{p=w} for $\G=\PGL_2$. 

Conversely, assuming Theorem \ref{p=w} holds, then for all $l\ge 0$ and $m\ge 0$, we have $$ \beta^l\gamma^m\in P_{2l+4m}(H^{4l+6m}(\M)).$$ Still by Proposition \ref{pwlift} this implies that there exists a lift $\widetilde{y} \in H^{2l+4m}(\cM)$ of $\beta^l\gamma^m$ such that $\widetilde{y}\eta^{3g-2-2l-2m} = 0$, and by Proposition \ref{rmksigma}, we can choose such a lift to be $\Sp(2g,\mathbb{Z})$ invariant; therefore, it is of the form $\widetilde{y}=\beta^l\gamma^m+\eta q(\alpha,\beta,\gamma,\eta)$ for some polynomial $q$. Then, by Poincaré duality on $\cM$ we have  $$\int_{\cM}\eta^{3g-2-2l-2m}(\beta^l\gamma^m+\eta q(\alpha,\beta,\gamma,\eta))x = 0 $$ for all $x\in H^*(\cM)$. Finally, since $\int_{\cM}\eta v = \int_Z v$ for all $v$ the proof is complete.
\end{proof}

\begin{remark} There is a straightforward generalization of the result above for $PGL_n$. We have universal generators for $H^*(\M^d_\Dol(\PGL_n))$, whose weights on the character variety side have been computed in \cite{hausel-villegas,shende} and the curious Hard Lefschetz theorem in $H^*(\M^d_\B(\PGL_n))$ has been proved in \cite{mellit}. 
	\end{remark}

\section{Equivariant integration on $\M$}

\subsection{Kalkman's formula}
Let $X$ be a semiprojective $\T$-manifold. The (formal)  \textit{equivariant integral} of a class $x\in H_\T^*(X)$ is then defined \cite{hausel-proudfoot} as a sum over the fixed-point components
\beq\label{equint}\oint_X x \coloneqq  \sum_{F\in \pi_0(X^\T)} \int_F \frac{x|_{F}}{e_\T(N_{F})}  \in \C[u,u^{-1}]\eeq where $e_\T(N_{F})$ is the equivariant Euler class of the normal bundle to $F\subseteq X$.
The equivariant integral thus is a  functional of degree $-2\dim X$ on $H_\T^*(X)$, taking values in Laurent polynomials in $u$. Sometimes we will allow the class $x$ to be in a completion of $H_\T^*(X)$, and in this case, the values will be in Laurent series in $u$.  When $X$ is compact, both sides of \eqref{equint} are well defined and according to \cite{BV,AB}, \eqref{equint} is an actual equality.

This operation appears in our calculations via \textit{Kalkman's residue formula} \cite{kalkman,lerman,edidin-graham}, which, in our context gives the following for the intersection numbers on $Z$:
\beq\label{kalkman} \int_Z \kappa_Z(x) =- \res_{u=0} \oint_\M x\,du, \eeq
where  $d_Z=\dim Z=6g-7$ and class $x\in H_\T^{2d_Z}(\M)$.

It is clear then that to approach \eqref{intersectionz}, we need a formula for the equivariant  integrals on $\M$, which is easy to manipulate. First we introduce a new notation for the cohomology classes on $\M$.

\subsection{Witten's notation for the cohomology classes of $\M$ and $\N$}
\label{wittennot}
We generalize slightly the construction of tautological classes in \S\ref{sec:generators}.

We recall that we defined in \eqref{defabps} universal classes as K\"unneth components of the second Chern class of the bundle $\text{End}(\bE))$. More generally,  we let the permutation group $S_2$ act on $\C[y]$ by $y\mapsto -y$ and take  $P\in \C[y]^{S_2}$, an even polynomial. Then there is a polynomial $\tilde{P}$  such that $P(y)=\tilde{P}(-y^2)$. We form the class $$\tilde{P}(\text{c}_2(\text{End}(\bE)))
\in H^*_\T(\M\times C),$$ and we introduce the following notation for its K\"unneth components:
$$\tilde{P}(\text{c}_2(\text{End}(\bE)))=P_{(2)} \otimes \omega + \sum_{i=1}^{2g} P_{(e_i)} \otimes e_i +P_{(0)} \otimes 1 \in H_\T^*(\M)\otimes H^*(C); $$ this defines a family of equivariant cohomology classes $P_{(2)}, P_{(e_i)}, P_{(0)}\in  H_\T^*(\M)$. 	
Note that, because of the compatibility of the universal bundle with the $\T$-action and the restriction to $\N$, this formalism allows us to define classes in $H^*(\M)$ and $H^*(\N)$. In particular, in which cohomology group lies e.g. $P_{(2)}$ will be determined by the context.   (This notation was introduced by Witten in \cite{Wi} for classes in $H^*(\N)$.)

It is easy to see that 
\begin{equation}\label{defclass}   [-y^2/2]_{(2)}=\alpha,\;  [y^2]_{(0)} = \beta ,\;[-y^2/4]_{(e_i)}= \psi_i,\;  [-y^4/4]_{(2)}=\alpha\beta+4\gamma,
\end{equation} where 
\begin{equation}\label{gammadef}
	\gamma = -2\sum_{i=1}^g\psi_i\psi_{i+g}.
\end{equation}
To show this, note that, for polynomials $P,Q\in \C[y]^{S_2}$, we have 
\[
[PQ]_{(2)} = P_{(2)}Q_{(0)} + P_{(0)}Q_{(2)}  
+  \sum_{i=1}^{g} \left(P_{(e_i)}Q_ {(e_{i+g})}- P_{(e_{i+g})}Q_ {(e_{i})}\right).                                                                                                                                                                 \] Then  we obtain \eqref{gammadef} when setting $P = y^2/2$ and ${Q=-y^2/2}$ in this formula.

\begin{remark} \label{rmktqall}
	In general, for any polynomial $P\in \mathbb{C}[y^2]$, we can write the contraction $P_{(2)}$ in terms of classes $\alpha$, $\beta$ and $\gamma$ in the following way:
	\begin{equation}\label{contraction2}
	P_{(2)} =    \frac{2\gamma -\alpha\beta}{\beta\sqrt{\beta}} P'(\sqrt{\beta})-\frac{2\gamma}{\beta}  P''(\sqrt{\beta}). 
	\end{equation}
	
	To show this, set $t = y^2$. Then by the Leibnitz rule we have 
	\[
	(t^k)'' = k(k-1) t^{k-2}(t')^2 + kt^{k-1}t''.
	\]
	Now by substituting $(t')^2\rightarrow -8\gamma$, $t'' \mapsto -2\alpha$, 
	$t\rightarrow \beta$ we deduce that if $Q(t) = P(\sqrt{t})$, then 
	\[
	P_{(2)} = -2\alpha Q'(\beta) -8\gamma Q''(\beta),
	\]
	where the derivatives of $Q$ are taken with respect to the variable $t$. The formula follows from the identities
	\[
	\frac{\D}{\D t} = \frac{1}{2y}\frac{\D}{\D y}, \ \ \ \frac{\D^2}{\D t^2} = \frac{1}{4y^2}\left(\frac{\D^2}{\D y^2} - \frac{1}{y}\frac{\D}{\D y}\right).
	\]

	This, in particular, implies that all intersection numbers involving the classes $\alpha$, $\beta$ and $\gamma$
	are encoded in the numbers of the form $\int_{\N}T_{(0)}\exp(Q_{(2)}).$ Indeed, we can choose for example $Q = -A y^2/2 - G y^4/4$, getting any class by subsequent derivatives with respect to the formal variables $A$ and $G$. 
	In rank $2$, this in particular gives all intersection numbers for the part of the cohomology which is invariant under the action of the symplectic group induced by the one on $H_*(C)$: see Proposition \ref{rmksigma} below.
	
\end{remark}

Using this notation, we can present a rather general integration formula on $\N$.
Such formulas were first presented by Thaddeus \cite{thaddeus1} (cf. \cite{Wi,JK,zagier,Sz}). We will provide a proof which uses Zagier's formalism \cite{zagier}.

\begin{theorem} Let $T \in \C[y^2]$ and $P\in y^2\C[y^2]$. Then
	\begin{equation}\label{stabint}
	\int_{\N} T_{(0)}\exp(P_{(2)}) =
	\thres_{y=0} \frac{2^{-1}T(y)P''(y)^g}{
		y^{2g-2}\left[\exp(P'(y))-\exp(-P'(y))\right] }
	\end{equation}
\end{theorem}

\begin{proof}
Changing variables via $Q(T) = y^2$ in Proposition 2 of \cite{zagier}, we obtain 
\begin{equation}\label{zagierinters}\int_\N f(\beta)e^{u(\beta)\alpha + w(\beta)\gamma^*} = \res_{y=0}\frac{(-4)^{g-1}f(y^2)u(y^2)^g}{y^{2g-2}\sinh(yu(y^2)+y^3w(y^2))}.
\end{equation}
Observe that substituting ${\gamma^* = \alpha\beta-2\gamma}$ in \eqref{contraction2} \footnote{Notice that formula (6) in \cite{zagier} matches the one in \cite{thaddeus1} at page 14, but the $\gamma$'s defined in the two papers differ by a sign. Therefore, we have to apply the formulas in \cite{zagier} by changing the sign of $\gamma$ wherever it appears.} one arrives at $$P_{(2)} = -\alpha P''(\sqrt{\beta}) + \gamma^*\left(P''(\sqrt{\beta})/\beta-P'(\sqrt{\beta})/\beta\sqrt{\beta}\right) $$ (note that the coefficient of $\gamma^*$ is a polynomial since $P$ is divisible by $y^2$). Finally, performing the substitutions $$ f(\beta)=T(\sqrt{\beta}),\ \ u(\beta) = -P''(\sqrt{\beta}), \ \ w(\beta) = P''(\sqrt{\beta})/\beta-P'(\sqrt{\beta})/\beta\sqrt{\beta}$$  in \eqref{zagierinters}, we obtain \eqref{stabint}. The discrepancy of a factor of $2^{2g}$ comes from the difference between the moduli of $SL_2$ and $PGL_2$ bundles.
\end{proof}

\subsection{The equivariant integration formula}
Now we are ready to fomulate one of our main results.
\begin{theorem} \label{equintm}
	\[  \oint_\M   T_{(0)}\exp(Q_{(2)})  =
\sum_{r\in\{0,u,-u\}}	\thres_{y=r}\frac{2^{-1}T\cdot(Q''-2u/(u^2-y^2))^g}
	{ u^{g-1}\left(e^{Q'}\frac{u-y}{u+y}-e^{-Q'}\frac{u+y}{u-y}\right)
		y^{2g-2}(u^2-y^2)^{g-1}}\,dy
	\]
\end{theorem}

\begin{proof}
	To calculate the equivariant integral as defined in \eqref{equint}, we need to list the components of the fixed point set  $\M^{\T}$ of the circle action and identify the corresponding normal bundles. This data may be found in \cite[\S4,5,6]{hausel-thaddeus2}, and thus we will brief here. There are two sorts of stable $\T$-fixed Higgs bundles $ (E,\Phi) $  (cf. beginning of \S\ref{intro1}):
	\begin{itemize}
		\item $E$ is stable and $\Phi=0$. This set of Higgs bundles forms a copy $\N\subset\M$ of the moduli of stable bundles, which extends to the embedding $T^*\N\subset\M$.
		\item $ E=L\oplus \Lambda L^{-1} $, $1\le\deg L\le g-1$, $\Phi|_{\Lambda L^{-1} }=0$, and 
		$\Phi|_{ L}$ is a nonzero map of line bundles $L\to \Lambda L^{-1} K$. As the degree of $\Lambda L^{-2} K$ is $2g-1-2\deg L$, such a map gives us a divisor on $C$, i.e. an element of $S^{2g-2\deg L-1}(C)$ of the corresponding symmetric product of the curve. It is easy to see that this correspondence is, in fact, a bijection, the list of fixed point set of this second type is $F_i\cong S^{2g-2i-1}(C)$, $i=1,\dots,g-1$. 
	\end{itemize}

We thus have	 
	\begin{equation}\label{equintMterms}
	\oint_\M   T_{(0)}\exp(Q_{(2)}) = \int_{\N}\frac{ T_{(0)}\exp(Q_{(2)})}{\eqeu(N_\calN)}
	+ \sum_{i=1}^{g-1}\int_{F_i}\frac{ T_{(0)}\exp(Q_{(2)})}{\eqeu(N_{F_i})}
	\end{equation}

	We start by describing a formula for multiplicative characteristic classes on $\N$.
	Let $B(t)$ be a formal power
	series and for a vector bundle $V$ denote by $B(V)$ the corresponding
	multiplicative characteristic class of the $V$: denoting the Chern
	roots of $V$ by $t_i,\,i=1,\dots,\mathrm{rank}(V)$, this class is
	the product $\prod_iB(t_i)$. 
	\begin{lemma} 
		Let $B(t)$ be a formal power series with $B(0)\neq1$. Then
		\begin{equation}
		\label{bm}
		B(T\N) =  \left[ B(0)B(y)B(-y)\right]^{g-1}_{(0)}
		\exp\left(\left[\widehat B(y)+\widehat B(-y)\right]_{(2)}\right),  
		\end{equation}
		where $\widehat B(t)$ satisfies $\prt t\widehat B(t) =-\log B(t)$.
	\end{lemma}
	\begin{proof}
		
		Denote by $\pi$ the projection $C\times\N\to\N$ and by
		$\omega$ the positive integral generator of the second cohomology of
		$C$. Then by \cite{zagier} we have,
		for $k>0$,
		\begin{equation}
		\label{abzag}
		s_k(T\N) = \pi_*((g-1)(s_k(\mathrm{Ad}(U))\omega-s_{k+1}(\mathrm{Ad}(U))/(k+1)).  \end{equation}
		Here $s_k$ is the power sum symmetric polynomial (in any number of
		variables), i.e. $s_k(b_1,b_2,\dots)=b_1^k+b_2^k+\dots$ 
		
		We write $B(T\N)$ as $\exp \left(\sum_i\log(B(t_i))\right)$, and
		then we apply the equality \eqref{abzag} to the sum in the
		exponential. The result is two terms, the first of which contributes
		a factor $B(\mathrm{Ad}(U))^{g-1}\cap 1$, while the second gives
		$\exp(\widehat B(\mathrm{Ad}(U))\cap\omega)$, where 1 and $\omega$ are, as usual,
		the positive integral generators of $H_0(C,\Z)$ and $H_2(C,\Z)$,
		respectively. The Chern roots of $\mathrm{Ad}(U)$ are $0,\pm y$, and using the
		notation introduced in \S\ref{wittennot}, we obtain
		\eqref{bm}.
	\end{proof}

	Now we can calculate the first term of the \eqref{equintMterms}.
	Observe that $\eqeu(T^*\N)^{-1}$ is a multiplicative class
	of $T\N$ corresponding to the function \[ B(t)=\Psi(t)\overset{\mathrm{def}}{=}1/(u-t) \] Then  
	\[ \widehat \Psi(t) = -(u-t)\log(u-t)-t \text{ with }\frac{\D}{\D t}\widehat \Psi(t)=\log(u-t)\text{ and }
	\frac{\D^2}{\D t^2}\widehat \Psi(t)=\frac{-1}{u-t}
	\]
	and
	\[ \frac1{\eqeu{(T^*\N)}}=\frac1{u^{g-1}}
	\left[\frac1{(u-y)(u+y)}\right]_{(0)}^{g-1}
	\exp\left( \left[\widehat\Psi(y)+\widehat\Psi(-y)\right]_{(2)}\right).\]
	We have 
	\[  \frac{\D}{\D y}\left( \Psi(y)+\Psi(-y)\right) =\log(u-y)-\log(u+y)
	\] \[ \frac{\D^2}{\D y^2}\left( \Psi(y)+\Psi(-y)\right)=
	\frac{-1}{u-y}+\frac{-1}{u+y}
	\]

	Now, combining this with \eqref{stabint} and \eqref{bm} shows that 
	\begin{equation}\label{contributionn}
	\int_{\N}\frac{ T_{(0)}\exp(Q_{(2)})}{\eqeu(T^*\calN)}=
	\frac{1}{2}\res_{y=0}\frac{T\cdot(Q''-2u/(u^2-y^2))^g}
	{ u^{g-1}\left(e^{Q'}\frac{u-y}{u+y}-e^{-Q'}\frac{u+y}{u-y}\right)
		y^{2g-2}(u^2-y^2)^{g-1}}\,dy.
	\end{equation}
	This calculates the first summand of \eqref{equintMterms} since, being $\N\subseteq \M$ a Lagrangian subvariety, we have $N_\N = T^*\N$.

	Now let $F_i \simeq S^{2g-2i-1}(C)$ for $i=1,\ldots, g$ be the other components of the fixed point set.
To evaluate 	$$\int_{F_i}\frac{ T_{(0)}\exp(Q_{(2)})}{\eqeu(N_{F_i})}\in \C(u), $$ first we compute $T_{(0)}|_{F_i}$ and $Q_{(2)}|_{F_i}$.  For this we define (c.f. \cite[\S 5]{hausel-thaddeus2}) universal classes $\eta\in H^2(F_i)$ and $\xi_i\in H^1(F_i)$ for $i=1\dots 2g$ by computing the first Chern class
of the universal divisor $\Delta\subset F_i\times C\cong S^{2g-2i-1}(C)\times C$ in K\"unneth decomposition:

$$c_1(\Delta)=(2g-2i-1)\omega+\sum_{i=1}^{2g} \xi_i e_i + \eta \in H^2(F_i\times C)= $$ $$=H^0(F_i)\times H^2(C)\oplus H^1(F_i)\times H^1(C)\oplus H^2(F_i)\times H^0(C).$$ We also define $\theta_i=\xi_i\xi_{i+g}$ for $i=1\dots g$ and $\theta=\sum_{i=1}^g \theta_i$. With these notations we have. 

\begin{lemma} We have \beq \label{restrictiont}T_{(0)}|_{F_i}=T(\eta-u)\in H^*_\T(F_i)\cong H^*(F_i)[u]\eeq and 
	\beq\label{restrictionq}Q_{(2)}|_{F_i}=(1-2i) Q^\prime(\eta-u)-\theta Q^{\prime \prime}(\eta-u)\in H^*_\T(F_i)\cong H^*(F_i)[u]\eeq
	\end{lemma}
\begin{proof} Recall from last displayed line of proof of \cite[(6.1)]{hausel-thaddeus2} that $$ -\text{c}_2(\End(\bE))|_{F_i}=( (1-2i)\omega +
\eta-u + \sum_{j=1}^{2g} \xi_j e_j)^2.$$ We get \eqref{restrictiont} immidiately and that $$ [y^{2k}]_{(2)}|_{F_i}=(1-2i)(2k)(\eta-u)^{2k-1}+\binom{2k}{2}(-2 \theta) (\eta-u)^{2k-2}.$$ This in turn yields \eqref{restrictionq} 
\end{proof}

Next we need a formula for the equivariant Euler class $e_\T(N_{F_i})$. 

\begin{lemma} \label{equeunfi} We have\bes e_\T(N_{F_i})&=& (2u-\eta)^{g+2i-2}   \exp\left(
	\frac{\theta}{\eta-2u}\right) (u)^{g-1} (u-\eta)^{-2i+g}  \exp\left(\frac{\theta}{\eta-u}\right)  
\\ && \cdot (\eta-u)^{2i-2+g} 	\exp\left(\frac{\theta}{u-\eta}\right)\\  &=&  u^{g-1} (-1)^g  (\eta-u)^{2g-2} (2u-\eta)^{g+2i-2}   \exp\left(
	\frac{\theta}{\eta-2u}\right) .\ees
	
	\end{lemma}

\begin{proof} We know that the tangent bundle of $\M$ restricted to $F_i$ can be computed as $$T_\M|_{F_i}\cong R^1\pi_*\left(\End_0(\bE)\stackrel{Ad({\boldsymbol{\Phi}})}{\to}\End_0(\bE)\otimes K_C\right),$$
where $\pi: F_i\times C\to F_i$ is the projection. To compute this first we have from \cite[(6.1)]{hausel-thaddeus2} 
$$\ch_\T(\End_0(\bE))=$$ $$=1+\exp\left( (1-2i)\omega +
\eta -u + \sum_{j=1}^{2g} \xi_j e_j \right)+\exp\left( (2i-1)\omega -
\eta+u - \sum_{j=1}^{2g} \xi_j e_j\right).$$
Thus we can compute \bes\ch_\T(T_\M|_{F_i})&=&-\ch_\T\left(R\pi_*\left(\End_0(\bE)\stackrel{Ad({\boldsymbol{\Phi}})}{\to}\End_0(\bE)\otimes K_C\right)\right)\\ &=&\pi_*\left(\ch_\T(\End_0(\bE)(-1+\exp(u)\ch(K_C)){\mathrm t \mathrm d}(C)\right)\\&=& \pi_*\left( \ch_\T(\End_0(\bE))  (-1+\exp(u)+(g-1)(1+\exp(u))\sigma)\right)\\ &=& \exp(2u)((g-1)\exp(-\eta)+\exp(-\eta)(2i-1-\theta)\\&&+\exp(u)((g-1)(1+\exp(-\eta))-\exp(-\eta)(2i-1-\theta) \\&& +((g-1)(1+\exp(\eta))-\exp(\eta)(2i-1+\theta) \\&&+ \exp(-u)((g-1)\exp(\eta)+\exp(\eta)(2i-1+\theta)
\\ &=& \exp(2u)((2i-2)\exp(-\eta)+\sum_{i=1}^g \exp(-\eta-\theta_i))\\&&+\exp(u)((g-1)-2i\exp(-\eta))-\sum_{i=1}^g \exp(-\eta-\theta_i)) \\&& +((g-1)-2i\exp(\eta))-\sum_{i=1}^g\exp(\eta+\theta_i)) \\&&+ \exp(-u)((2i-2)\exp(\eta)+\sum_{i=1}^g \exp(\eta+\theta_i)).\ees
The four lines give the contributions from the four weight spaces of the $\T$-action on $T_\M|_{F_i}$: weight $2$, weight $1$, weight $0$ and weight $-1$. The $0$ weight space corresponds to the tangent space $T_{F_i}<T_{\M}|_{F_i}$. Removing it will yield the normal bundle $N_{F_i}$ of $F_i$ in $\M$. Thus \bes\ch_\T(N_{F_i}) &=& \exp(2u)((2i-2)\exp(-\eta)+\sum_{i=1}^g \exp(-\eta-\theta_i))\\&&+\exp(u)((g-1)-2i\exp(-\eta))-\sum_{i=1}^g \exp(-\eta-\theta_i)) 
 \\&&+ \exp(-u)((2i-2)\exp(\eta)+\sum_{i=1}^g \exp(\eta+\theta_i)).
\ees Formal computation now gives the Lemma.
\end{proof}

	The final ingredient is  due to Zagier \cite[(7.2)]{thaddeus}.  For any power series
	$A(x)\in\C[[x]]$ and $B(x)\in\C[[x]]$  we have \[\int_{F_i} A(\eta)\exp(B(\eta)\theta)=\res_{x=0} \left\{
	\frac{A(x)\left(1+x B(x)\right)^g}{
		x^{2g-2i}}\right\}dx.\] 
	
	So we can proceed to compute with $y=x-u$ \begin{multline*} \int_{F_i}\frac{ T_{(0)}\exp(Q_{(2)})}{\eqeu(N_{F_i})}=\int_{F_i} 
	\frac{T(\eta-u)\exp\left((1-2i) Q^\prime(\eta-u)-\theta Q^{\prime \prime}(\eta-u)\right)}{u^{g-1} (-1)^g  (\eta-u)^{2g-2} (2u-\eta)^{g+2i-2}   \exp\left(
		\frac{\theta }{ \eta-2u}\right)}\\= 
	\frac{(-1)^g}{u^{g-1}}\int_{F_i} \frac{T(\eta-u)\exp\left((1-2i) Q^\prime(\eta-u)\right)}{(\eta-u)^{2g-2} (2u-\eta)^{g+2i-2} }\exp\left(\left(-Q^{\prime \prime}(\eta-u)+	\frac{1 }{ 2u-\eta}\right)\theta\right) \\ 
	= 	\frac{(-1)^g}{u^{g-1}} \res_{x=0} \frac{T(x-u)\exp\left((1-2i) Q^\prime(x-u)\right)\left(1+x \left(-Q^{\prime \prime}(x-u)+	\frac{1 }{ 2u-x}\right) \right)^g\D x}{(x-u)^{2g-2} (2u-x)^{g+2i-2} x^{2g-2i}}
	  \\ = 	u^{1-g} \res_{y=-u}  \frac{T(y)\exp\left((1-2i) Q^\prime(y)\right)}{y^{2g-2} (u-y)^{g+2i-2}(u+y)^{g-2i} }\left(Q^{\prime \prime}(y)-	\frac{1 }{ u+y}-	\frac{1 }{ u-y} \right)^g
	\D y
	\end{multline*}
	
	We note that the right hand side of this expression has no pole at $y=-u$ for $i\geq g$ 
	thus the residue vanishes and so we have \begin{multline*} \sum_{i=1}^{g-1}  \int_{F_i}\frac{ T_{(0)}\exp(Q_{(2)})}{\eqeu(N_{F_i})}\\  = \sum_{i=1}^{g-1}	u^{1-g} \res_{y=-u} \left\{ \frac{T(y)\exp\left((1-2i) Q^\prime(y)\right)}{y^{2g-2}(u-y)^{g+2i-2} (y+u)^{g- 2i} }\left(Q^{\prime \prime}(y)-	\frac{1 }{ u+y}-	\frac{1 }{ u-y} \right)^g
	\right\} d y\\ 
	=\sum_{i=1}^{\infty}	u^{1-g} \res_{y=-u} \left\{ \frac{T(y)\exp\left((1-2i) Q^\prime(y)\right)}{y^{2g-2}(u-y)^{g+2i-2} (y+u)^{g- 2i} }\left(Q^{\prime \prime}(y)-	\frac{1 }{ u+y}-	\frac{1 }{ u-y} \right)^g
	\right\} d y\\  =	\res_{y=-u} \frac{ T\cdot(Q''-2u/(u^2-y^2))^g} 
	{ u^{g-1}\left(e^{Q'}\frac{u-y}{u+y}-e^{-Q'}\frac{u+y}{u-y}\right)
		y^{2g-2}(u^2-y^2)^{g-1}}\,dy. \end{multline*}

Finally we note that the expression in this residue is odd in $y$ and so will give the same result at $y=u$.
Thus we can conclude that $$\sum_{i=1}^{g-1} \int_{F_i}\frac{ T_{(0)}\exp(Q_{(2)})}{\eqeu(N_{F_i})}=\sum_{r\in \{u,-u\}}  \frac{1}{ 2}\res_{y=r} \frac{  T\cdot(Q''-2u/(u^2-y^2))^g\D y} 
	{ u^{g-1}\left(e^{Q'}\frac{u-y}{u+y}-e^{-Q'}\frac{u+y}{u-y}\right)
		y^{2g-2}(u^2-y^2)^{g-1}} $$

Together with \eqref{contributionn} this completes the proof of the Theorem.
\end{proof}

\subsection{The $\Sp(2g,\mathbb{Z})$ action}

Note that the orientation preserving index-2 subgroup of the mapping class  group of the Riemann surface underlying $C$ acts on $H^*(C)$ via\footnote{We note that our notation for $\Gamma$ and $\Sigma$ are swapped from the notation in \cite{hausel-thaddeus2}, this is to be more in line with  the notation of $\Gamma$ used in \cite{hausel-thaddeus3}.} $$\Sigma\coloneqq \Sp(2g,\Z)$$ acting on $H^0(C)$ and $H^2(C)$ trivialy and by preserving the symplectic structure on $H^1(C)$ given by the intersection pairing.

We see that $\Sigma$ acts on $H^3(\M)\cong H^1(C)$, where we identify $H^3(\M)$ with $H^1(C)$ by mapping $\psi_i$ to $e_i$ and letting $\Sigma$ act trivially on $\C[u]$. This will also yield an action of $\Sigma$ on $H^3_\T(\M)\cong H^3(\M)\otimes \C[u]$, where the identification is again done by mapping the equivariant $\psi_i$ to the non-equivariant one. Finally we will let $\Sigma$ act on $H^2_\T(\M)$ and $H^4_\T(\M)$ and so on $\alpha$ and $\beta$ trivially.

\begin{proposition}\label{rmksigma}  
The action defined above induces an action of $\Sigma$ on $H^*_{\mathbb{T}}(\M)$ by $\mathbb{C}[u]$-algebra automorphisms, and on $H^*(Z)$ by ring automorphisms. 
 The perverse filtrations on $H^*(\mathcal{M})$ and on $H^*(Z)$ are invariant under this action.
\end{proposition}
 \begin{proof}
 We observe that the equivariant intersection numbers $$\oint_\M \alpha^k\beta^l\prod_{i=1}^{2g}\psi_i^{m_i}\in \C(u)$$ are invariant with respect to this action of $\Sigma$. This follows from the fact that in \eqref{equint} the expression for the equivariant Euler class $e_\T(T^*\N)$ coming from \eqref{abzag}  is invariant under $\Sigma$. Analogously, the formula in Lemma~\ref{equeunfi} is $\Sigma$-invariant. 
 
Thus we deduce that the intersection numbers on $\N$ and $F_i$ are invariant under $\Sigma$ as the  $\Sigma$ action on $H^*(\N)$ and $H^*(F_i)$ comes from the mapping class group of $C$. It follows that the ideal of relations among the universal generators $\alpha,\beta,\psi_i$ of $H^*_\T(\M)$ is invariant under $\Sigma$, thus getting an action of $\Sigma$ on $H^*_\T(\M)$ by $\C[u]$-algebra automorphisms extending the one on $H^2_\T(\M)\oplus H^3_\T(\M)\oplus H^4_\T(\M)$ defined above.

In turn, by \eqref{kalkman} we get that all intersection numbers on $Z$ will be invariant by  $\Sigma$, which will yield a $\Sigma$-action on $H^*(Z)$ by ring automorphisms. Thus, since $\int_Zx = \int_{\cM}\eta x$, we see that the perverse filtration on $H^*(\cM)$ will be $\Sigma$ invariant, and finally also the one on $H^*(\M)$ by Proposition \ref{pwlift}. \end{proof}

Since the invariant part $H^*_\T(\M)^\Sigma$ is generated by $\alpha,\beta,\gamma$, we conclude that all intersection numbers are encoded in the $\Sigma$-invariant ones
$\int_\M \alpha^k\beta^l \gamma^m$. Moreover, thanks to Proposition \ref{rmksigma}, it is enough to prove P=W for the classes $\alpha$, $\beta$ and $\gamma$.

\subsection{The Bethe pole}
The formula in Theorem \ref{enumerativepw} is an explicit statement, which ought to follow from Theorem \ref{equintm}, but this calculation rather difficult to perform. The main reason is that it involves cancellation of two rather different terms: the residues of a differential form at $y=0$ and at $y=u$. In this section we mitigate this problem, and rewrite our statement in a completely local form.

Let us consider expression in Theorem \ref{equintm}. 
First we observe that the expression in parenthesis may be rewritten as follows $$
e^{Q'}\frac{u-y}{u+y}-e^{-Q'}\frac{u+y}{u-y}=\frac{(u-y\tanh(Q'/2))(u-y\coth(Q'/2))}{u^2-y^2}.
$$

The equivariant integral thus is the sum of the residues of the expression
\begin{equation}\label{exprtanh}
\Omega(u,y)=	\frac{2^{-2}T\cdot((u^2-y^2)Q''-2u)^g}
{u^{g-1}\sinh(Q')(u-y\tanh(Q'/2))(u-y\coth(Q'/2))y^{2g-2}(u^2-y^2)^{2g-2}}
\end{equation}
at $y=0$ and $y=\pm u$. Let us find all the other poles of this form. The following statement is an  easy consequence of the implicit function theorem. Its proof will be omitted.

\begin{proposition}\label{polesomega}
	a. Let us assume that $T(0)=1$ and $Q=Ay^2+P(y)$, where $A$ is  a nonzero constant, P is an even polynomial of degree $>2$. Then for $u$ sufficiently small, there is a $c>0$ such that the form $\Omega$ in \eqref{exprtanh} has the following poles inside the unit disc:
	\begin{itemize}
		\item at $y=0$, of order $2g-1$
		\item at $y=\pm u$ of order $2g-2$
		\item simple poles at $y=\pm b_0$, where $|b_0-\sqrt{u}/A|<cu^2$ for a constant $c$.
	\end{itemize}
	
	b. If $P=0$, i.e. $Q=Ay^2$,  then we can describe all remaining poles of $\Omega(u,y)$: these are  simple poles at $y=b_n$, $n\in\Z\setminus\{0\}$, where $|b_n-\pi i n|<cu$.
\end{proposition}

\begin{remark}
When $P=0$, the numbers $ B_u=\{b_n,\,n\in\Z\}$ are  the  solutions of the Bethe Ansatz for the Yang-Yang model \cite{YY}. This was a crucial observation of \cite{MNSh}, who also studied a special case of these equivariant integrals in order to find a \textit{regularized volume} of the Higgs moduli. Their formula, obtained using mathematically nonrigorous methods, is an infinite sum over $B_u$, and can be easily recovered by applying the Residue Theorem to Theorem \ref{equintm}.
\end{remark}

The most important part of Proposition \ref{polesomega} is that for  $Q=Ay^2+P(y)$, and $u$ sufficiently small, the only poles of $\Omega(u,y)$ in the disc $ \{y:\,|y|<2u\} $ are the poles $y=0,-u,u$. In this case, we have an integral representation
\[  \res_{y=0,\pm u} \Omega(u,y)\,\D y= \int_{|y|=2\epsilon}\Omega(u,y)\,\D y
\]
and thus
\[  \res_{u=0}\res_{y=0,\pm u}\Omega(u,y)\,\D y\,\D u  = \int_{|u|=\epsilon}\int_{|y|=2\epsilon}\Omega(u,y)\,\D y\,\D u . 
\]

To perform the last double integral, we must first fix the magnitude of $u$ to be equal to $\epsilon$, and then compute the $y$-residues inside the circle of the $y$-plane of radius $2\epsilon$. However, as the last expression is an integral over a product of circles, we can apply Fubini's theorem and write our integral as
\[   \int_{|u|=\epsilon}\int_{|y|=2\epsilon}\Omega(u,y)\,\D y\,\D u    = \int_{|y|=2\epsilon} \int_{|u|=\epsilon}\Omega(u,y)\,\D u\,\D y.
\]

Finally, we convert this last integral into residues again. Now we fix a value of $y$ of
magnitude $2\epsilon$ and look for poles of our form inside the circle
of radius $\epsilon$ in the $u$-plane. Again, studying the
denominator, we see that we have a pole at $u=0$, but the factor $u^2-y^2$
now does not contribute. We have a pole, however, at the point
$u=y\tanh(Q'/2)$, because this is of order $u\sim
4\epsilon^2\ll\epsilon$. This leads to the following residue identity:
\[  \res_{u=0}\res_{y=0,\pm u}\Omega(u,y)\D y\, \D u\,  =
\res_{y=0}\left[\res_{u=0}+\res_{u=y\tanh(Q'/2)}\right]
\Omega(u,y)\D u\, \D y.
\]

Combining this with \eqref{exprtanh}, we arrive at the following formula for integration on $Z$:
\begin{multline}\label{splitresidue}\int_Z   T_{(0)}\exp(Q_{(2)})  =
-\res_{y=0}\left[\res_{u=0}+\res_{u=y\tanh(Q'/2)}\right]\\
\left( \frac{2^{-2}T\cdot((u^2-y^2)Q''-2u)^g\,\text{d}u\,\text{d}y}
{u^{g-1}\sinh(Q')(u-y\tanh(Q'/2))(u-y\coth(Q'/2))y^{2g-2}(u^2-y^2)^{2g-2}}\right).
\end{multline}

The second term may be further simplified since the pole of $\Omega(u,y)\, du$ at  $u=y\tanh(Q'/2)$ is simple, and
thus to calculate the residue, we simply perform the substitution $u=y\tanh(Q')$. Using the
identities
\[  \sinh(a)(\tanh(a/2)-\coth(a/2))=-2,\quad
\tanh^2(a/2)-1=\frac{-1}{\cosh^2(a/2)}
\]
and $\sinh(a)=2\sinh(a/2)\cosh(a/2)$, the denominator turns into
\[  -2y\cdot y^{g-1}\tanh^{g-1}(Q'/2)\cdot y^{2g-2}\cdot y^{2(2g-2)}\cosh^{-2(2g-2)}(Q'/2),
\]
while the numerator is 
\[   2^{-2} T(u=y\tanh(Q'/2))\cosh^{-2g}(Q'/2)(-y^2Q''-y\sinh(Q'))^g.
\]
Canceling the similar factors we arrive at
\begin{multline} \res_{y=0}\frac{ 2^{-2}
  T(u=y\tanh(Q'/2))\cdot\cosh^{-2g}(Q'/2)\cdot(-y^2Q''-y\sinh(Q'))^g\, \text{d}y}
{-2y\cdot y^{g-1}\tanh^{g-1}(Q'/2)\cdot y^{2g-2}\cdot y^{2(2g-2)}\cosh^{-2(2g-2)}}=
\\
 \res_{y=0}\frac{ -2^{-3}
  T(u=y\tanh(Q'/2))\cdot\cosh^{2g-4}(Q'/2)\cdot(-yQ''-\sinh(Q'))^g\, \text{d}y}
{\tanh^{g-1}(Q'/2)\cdot y^{6g-6}}.
\end{multline}
Thus formula (\ref{splitresidue}) yields the following result:
\begin{proposition}
  \label{otherres} We have the following formula for the intersection numbers of $Z$: 
  \begin{multline}
    \label{twores}
\int_Z
T_{(0)}\exp(Q_{(2)})=\\
\thres_{y=0}\thres_{u=0}
\frac{-2^{-2}T\cdot((u^2-y^2)Q''-2u)^g\,\emph{d}u\,\emph{d}y}
{u^{g-1}\sinh(Q')(u-y\tanh(Q'/2))(u-y\coth(Q'/2))y^{2g-2}(u^2-y^2)^{2g-2}}+\\
 \thres_{y=0}\frac{ 2^{-3}
  T(u=y\tanh(Q'/2))\cdot\cosh^{2g-4}(Q'/2)\cdot(-yQ''-\sinh(Q'))^g\, \emph{d}y}
{\tanh^{g-1}(Q'/2)\cdot y^{6g-6}}.
  \end{multline}
\end{proposition}

Notice that if the polynomial $T$ is divisible by $\eta^{2g-2}$, then the first residue at $u=0$ vanishes. In particular, if $k\le g-1$, we obtain
\begin{corollary}\label{zresidue} Define $$R_{g,k}^Q(y) \coloneqq  (-1)^{g}2^{-3}\left(\frac{\emph{sinh}(Q'/2)}{y}\right)^{2g-2k-2}\emph{cosh}^{2k-2}(Q'/2)\left(Q''+\frac{\emph{sinh}(Q')}{y}\right)^g,$$ then, for $1\le k \le g-1$,
 \begin{equation}\label{orig}\int_Z \eta^{3g-3-2k}T_{(0)}\emph{exp}(Q_{(2)})=\thres_{y=0}\frac{\emph{d} y}{y^{4k-1}}T (u=y \cdot \emph{tanh}(Q'/2))R_{g,k}^Q(y). \end{equation}
 \end{corollary}

 \subsection{Defects and the order of the pole}

 We denote simply by $R_{g,k}(y)$ the function appearing in \eqref{orig} when we set $Q = -Ay^2/2-Gy^4/4$. Notice that it is a holomorphic function around $y=0$.
 Then \eqref{defclass} and Corollary \ref{zresidue} tell us that $$\int_Z\eta^{3g-3-2k}\alpha^i(4\gamma)^j\beta^m\eta^n\text{exp}(Q_{(2)}) = $$ $$ \res_{y=0}\frac{\D y}{y^{4k-1}}y^{2m} \partial_A^i(\partial_G-y^2\partial_A)^j(y\ \! \text{tanh}(Q'/2))^nR_{g,k}(y).$$
 Let $\widetilde{R}_{g,k} = R_{g,k}(y,G = y^{-2}\widetilde{G})$: notice that it is still a holomorphic function in $y$. Then, since $\partial_G = y^2\partial_{\widetilde{G}}$, we get $$\int_Z\eta^{3g-3-2k}\alpha^i(4\gamma)^j\beta^m\eta^n\text{exp}(Q_{(2)}) = $$ $$\res_{y=0}\frac{\D y}{y^{4k-1}}y^{2m+2n+2j}\partial_A^i(\partial_{\widetilde{G}}-\partial_A)^j(\tanh(Q'/2)/y)^n\widetilde{R}_{g,k}(y) $$
 and in this formula, we are taking the residue of a meromorphic function with a pole at $y=0$ of order $4k-1-2m-2n-2j$. Since the order of the pole of the sum two meromorphic functions is at most the maximum between the orders of the two poles, we are brought to make the following definition.
 
 \begin{definition}\label{defect}
 We define the \emph{defect} of a monomial in $\alpha,\beta,\gamma$ and $\eta$ via the assignment $$\text{def}( \alpha )= 0, \ \text{def}(\beta) =  \text{def}(\gamma) =  \text{def}(\eta) = 2 $$ and extending it by multiplicativity. We also define the defect of any polynomial in $\alpha,\beta,\gamma$ and $\eta$ to be the \emph{minimum} of the defects of its monomials.
 \end{definition}

Notice that if $x$ is a class in $\alpha, \beta$ and $\gamma$, so that it can be seen as a class $x\in H^i(\mathcal{M})$, then we immediately verify $$\text{def}(x) = i - \text{wt}(x)$$ where $\text{wt}(x)$ is the weight of $x$.

By Theorem \ref{enumerativepw} we immediately get the following.

\begin{corollary}
If for every $x\in H^*(Z)$ we can find $y\in H^*(Z)$ with $$\int_Z\eta^{3g-3-\emph{def}(x)}
(x+\eta y)r = 0, \ \emph{for all } r\in H^*(Z)$$ then P=W holds for $\mathcal{M}$.
\end{corollary}

\section{The matrix problem for the top defect pairing}

\subsection{The $\beta^k$ classes}

In this section we will prove the following.

\begin{theorem}\label{mainbeta}
	Let $g\ge 2$ and $1\le k \le g-1$. There exists a unique class $F_k\in H^*(Z)$ such that, for all $P\in H^*(Z)$ with $\emph{def}(P) = 2k-2$, we have \begin{equation}\label{eq}\int_Z\eta^{3g-3-2k}(\beta^k + \eta F_k)P = 0.\end{equation}
\end{theorem}

In the entire section, we will always assume that $1\le k \le g-1$.

In order to prove Theorem \ref{mainbeta}, we need to show that the pairing $$(F,P)\mapsto \int_Z\eta^{3g-3-2k}FP, $$ $$ F\in H^{4k}(Z), \ \text{def}(F) = 2k, \ P \in H^{6g-8}(Z), \ \text{def}(P) = 2k-2$$ where we are taking defect-homogeneous $F$ and $P$, is degenerate; moreover, we need an element of its kernel to be of the form $\beta^k+\eta F_k$ for some $F_k$.

To compute the matrix of the pairing, we choose the following basis of the $F$ classes $$F_{k,a_1,n_1} \coloneqq  \beta^{a_1-n_1}(4\gamma)^{n_1}\eta^{k-a_1}\alpha^{k-a_1-n_1}, \text{ with }\begin{cases} 0\le n_1\le a_1, \\a_1+n_1 \le k,\end{cases}$$

and the following basis of the $P$-classes $$P_{k,a_2,n_2} \coloneqq  \beta^{a_2-n_2}(4\gamma)^{n_2}\eta^{k-1-a_2}\alpha^{3g-3-k-a_2-n_2}, \text{ with }   0\le n_2\le a_2 \le k-1.
$$
The coefficients have been chosen in order to make computations easier later.

We then perform a column operation and define the matrix $M_{k}$ of the pairing as follows: \begin{equation}\label{ratiomat}(M_{k})_{a_1,n_1}^{a_2,n_2} \coloneqq \frac{(-2)^{k-a_1}}{(k-a_1-n_1)!n_1!}\frac{\int_Z\eta^{3g-3-2k}F_{k,a_1,n_1}P_{k,a_2,n_2}}{\int_Z\eta^{3g-3-2k}F_{k,k,0}P_{k,a_2,n_2}}\end{equation}

Therefore Theorem \ref{mainbeta} is equivalent to finding a vector in the kernel of $M_{k}^T$ whose coefficient corresponding to the term $\beta^k$ (i.e. the row indexed by $(a_1,n_1) = (k,0)$) is nonzero.

\begin{lemma}\label{explicit}
	We have
	$$(M_{k})_{a_1,n_1}^{a_2,n_2} = \binom{3g-3-a_1-n_1-a_2-n_2}{k-a_1-n_1}\binom{g-n_2}{n_1}. $$
\end{lemma}

\begin{proof}
	We use the formula of Corollary \ref{zresidue} choosing the polynomial $Q = -A y^2/2 - Gy^4/4$.
	First of all we perform the computations for $$\widetilde{F}_{k,a_1,n_1} \coloneqq  \beta^{a_1-n_1}(\alpha\beta+4\gamma)^{n_1}\eta^{k-a_1}\alpha^{k-a_1-n_1}, $$ $$\widetilde{P}_{k,a_2,n_2} \coloneqq  \beta^{a_2-n_2}(\alpha\beta+4\gamma)^{n_2}\eta^{k-1-a_2}\alpha^{3g-3-k-a_2-n_2}.$$
	Let us write $\widetilde{G} = G y^2$, so that $\partial_G = y^2\partial_{\widetilde{G}}$. Then by Corollary \ref{zresidue} we have  $$\int_Z\eta^{3g-3-2k}\widetilde{F}_{k,a_1,n_1}\widetilde{P}_{k,a_2,n_2} =  $$ $$\partial_A^{3g-3-a_1-a_2-n_1-n_2}\partial_{\widetilde{G}}^{n_1+n_2}\res_{y=0}\frac{\D y}{y}\left(\frac{\text{tanh}(-A y/2 - \widetilde{G}y/2)}{y}\right)^{2k-1-a_1-a_2}\widetilde{R}_{g,k}(A,\widetilde{G};y)$$ where $\widetilde{R}_{g,k}(A,\widetilde{G};y) \coloneqq  R_{g,k}(A,\widetilde{G}/y^2;y)$.
	We see that the form we are taking the residue of has a simple pole, thus its residue is computed by evaluating at $y=0$. Now $$\widetilde{R}_{g,k}(A,\widetilde{G};0) = 2^{2k-1-g}(A+\widetilde{G})^{2g-2k-2}(A+2\widetilde{G})^g$$ and the tanh factor gives $(-A/2-\widetilde{G}/2)^{2k-1-a_1-a_2}$. Therefore we have $$\int_Z\eta^{3g-3-2k}\widetilde{F}_{k,a_1,n_1}\widetilde{P}_{k,a_2,n_2}= $$ $$=(-1)^{a_1+a_2+1}2^{-g+a_1+a_2}\partial_A^{3g-3-a_1-a_2-n_1-n_2}\partial_{\widetilde{G}}^{n_1+n_2}(A+\widetilde{G})^{2g-3-a_1-a_2}(A+2\widetilde{G})^g$$ (notice that the integral does not depend on $k$).  Then using the classes $4\gamma$ in the definitions of $F_{k,a_1,n_1}$ and $P_{k,a_1,n_1}$ amounts to change variable $B = A+\widetilde{G}$, so that the formula becomes  $$\int_Zu^{3g-3-2k}F_{k,a_1,n_1}P_{k,a_2,n_2}= $$ $$=(-1)^{a_1+a_2+1}2^{-g+a_1+a_2}\partial_B^{3g-3-a_1-a_2-n_1-n_2}\partial_{G}^{n_1+n_2}B^{2g-3-a_1-a_2}(B+G)^g =$$ $$= (-1)^{a_1+a_2+1}2^{-g+a_1+a_2}(3g-3-a_1-a_2-n_1-n_2)!(n_1+n_2)!\binom{g}{n_1+n_2},$$ dividing out and using the coefficients of \eqref{ratiomat} we obtain
	$$(M_{k})_{a_1,n_1}^{a_2,n_2} = \binom{3g-3-a_1-n_1-a_2-n_2}{k-a_1-n_1}\binom{g-n_2}{n_1} $$
	and the proof is complete.
\end{proof}

Lemma \ref{explicit} allows us to relate the matrix $M_k$ with a particular evaluation operator on polynomials in two variables.

\begin{corollary}\label{polyev} 
	
	Let $v=(v_{a_1,n_1})_{\substack{0\le n_1\le a_1 \\ n_1+a_1\le k}}$ be a row vector, and let $$p_v(X,Z) \coloneqq  \sum_{\substack{0\le n_1\le a_1 \\ a_1+n_1\le k}}v_{a_1,n_1}\binom{3g-3-a_1-n_1-Z}{k-a_1-n_1}\binom{g-X}{n_1} $$
	Then $ v M_{k} = (p_v(n_2,a_2+n_2))_{0\le n_2\le a_2\le k-1}$.
\end{corollary}

We are now ready to find a vector in the kernel of $M_{k}^T$.

\begin{proposition}\label{firstsolution}
	Let $$v_{k}^{a_1,n_1} \coloneqq  \thres_{t=0}\frac{(-1)^{k-a_1-n_1}(1+t)^{g-k-2}(1+2t)^{g-n_1}\emph{d} t}{t^{a_1-n_1+1}}, \ \ 0\le n_1\le a_1, \  a_1+n_1\le k.$$ Then $$\sum_{\substack{0\le n_1\le a_1 \\ a_1+n_1\le k}}v_{k}^{a_1,n_1}(M_{k})_{a_1,n_1}^{a_2,n_2} = 0$$ for all $a_2,n_2$ with $0\le n_2\le a_2\le k-1$. In particular, the vector $v_k$ is in the kernel of $M_{k}^T$.
\end{proposition}

\begin{proof}
	
	Define $$V_{g,k}(X,Z;v,w) \coloneqq  (1-v)^{-3g+k+2+Z}(1+w)^{g-X},$$ then we see that $\binom{3g-3-a_1-n_1-Z}{k-a_1-n_1}\binom{g-X}{n_1} = \res_{v,w=0}\frac{V_{g,k}(X,Z;v,w)\D v\D w}{v^{k-a_1-n_1+1}w^{n_1+1}}$, therefore 
	\begin{equation}\label{twovartop}
	(M_{k})_{a_1,n_1}^{a_2,n_2} = \res_{v,w=0}\frac{V_{g,k}(n_2,a_2+n_2;v,w)\D v\D w}{v^{k-a_1-n_1+1}w^{n_1+1}}.
	\end{equation}
	
	Now notice that \begin{equation}\label{master}(1+t)^{g-k-2}(1+2t)^gV_{g,k}(X,Z;-t,\frac{t^2}{1+2t}) = (1+t)^{Z-2X}(1+2t)^X\coloneqq  K(X,Z;t).\end{equation} For $0\le n_2\le a_2 \le k-1$, the values of $(X,Z) = (n_2,n_2+a_2)$ belong to $$\Gamma_k \coloneqq  \{ (X,Z) \in \mathbb{Z}^2 \ | \ X\ge 0, \ 2X\le Z \le X+k-1\}.$$ Notice that for $(X,Z)\in \Gamma_k$, $K(X,Z;t)$ is a polynomial in $t$ of degree $Z-X\le k-1$; therefore, by defining \begin{equation}\label{polydef} p_k(X,Z) \coloneqq  \res_{t=0}\frac{\D t}{t^{k+1}}K(X,Z;t) \end{equation} we find that $p_k(X,Z) = 0$ on $\Gamma_k$. Then, by taking the Taylor series of the left hand side of \eqref{master}, we have $$\res_{t=0} \frac{\D t}{t^{k+1}}(1+t)^{g-k-2}(1+2t)^gV_{g,k}(X,Z;-t,\frac{t^2}{1+2t}) = $$
	$$ = \sum_{0\le l,n \le \infty}\res_{t=0}\frac{(1+t)^{g-k-2}(1+2t)^g}{t^{k+1}}(-t)^l\left(\frac{t^2}{1+2t}\right)^n \res_{v,w=0}\frac{V_{g,k}(X,Z;v,w)\D v \D w}{v^{l+1}w^{n+1}} = $$ \begin{equation}\label{taylor} = \sum_{l+2n\le k}\res_{t=0}\frac{(1+t)^{g-k-2}(1+2t)^{g-n}}{t^{k-l-2n+1}}\res_{v,w=0}\frac{V_{g,k}(X,Z;v,w)\D v \D w}{v^{l+1}w^{n+1}} = p_k(X,Z)\end{equation} the last equality being formula \eqref{master}. Since $p_k(n_2,a_2+n_2) = 0$ for $0\le n_2\le a_2\le k-1$, we conclude thanks to Corollary \eqref{polyev} by substituting $l = k-a_1-n_1$ in the sum \eqref{taylor}.
\end{proof}

From this we can find the solution to Equation \eqref{eq}, and show that such solution is unique.

\begin{theorem}[Lowest defect]\label{topperv}
	The solution to Equation \eqref{eq} is $$\beta^k+uF_{k} = \frac{1}{p_k(g,3g-k-2)}\thres_{t=0}\frac{\emph{d} t}{t^{k+1}}(1+\beta t)^{g-k-2}(1+2\beta t)^ge^{2\eta t\left(\alpha -\frac{4\gamma t}{1+2\beta t}\right)} $$
\end{theorem}

\begin{proof}
	By Proposition \ref{firstsolution} and by (\ref{ratiomat}) we have that setting $$\widetilde{F}_k \coloneqq  \res_{t=0}\sum_{l+2n\le k}\frac{(-1)^{l}(1+t)^{g-k-2}(1+2t)^{g-n}(-2)^{l+n}\D t}{t^{k-l-2n+1}n!l!}\beta^{k-l-2n}(4\gamma)^{n}\eta^{l+n}\alpha^{l}= $$ $$ \res_{t=0}\frac{(1+t)^{g-k-2}(1+2t)^g\beta^k\D t}{t^{k+1}}\sum_{0\le l,n\le \infty} \frac{1}{n!l!}\left(\frac{2t\alpha \eta}{\beta}\right)^l\left(\frac{-8t^2\gamma \eta}{\beta^2(1+2t)}\right)^n  =$$ $$=\res_{t=0}\frac{\D t}{t^{k+1}}(1+\beta t)^{g-k-2}(1+2\beta t)^ge^{2\eta t\left(\alpha -\frac{4\gamma t}{1+2\beta t}\right)} $$
	then we have $\int_Zu^{3g-3-2k}\widetilde{F}_kP=0$ for all $P$ with defect $2k-2$. We see that the coefficient of the term $\beta^k$ in $\widetilde{F}_k$ is $$p_k(g,3g-k-2) = \res_{t=0}\frac{\D t}{t^{k+1}}(1+t)^{g-k-2}(1+2t)^g$$ which is clearly positive for $g\ge k+2$ since all exponents are, while for $g=k+1$ we have ${p_k(k+1,2k+1) = 2^{k+1}-1}$, still positive for $k\ge 1$. Thus we can in any case divide out and obtain the result.
\end{proof}

\begin{remark}\label{pkgenerate}
	Define $D_Z$ and $D_{X,-1}$ to be respectively the operators such that $$D_Zf(Z) = f(Z+1)-f(Z), \ \ D_{X,-1}f(X) = f(X-1)-f(X) $$ for any function $f$. Then we immediately see that $$D_ZK(X,Z;t) = tK(X,Z;t), \ \ D_{X,-1}K(X,Z;t) = \frac{t^2}{1+2t}K(X,Z;t).$$ This means that we can also prove Proposition \ref{firstsolution} by using the Newton interpolation formula $$\sum_{l+2n\le k}(-1)^lD_Z^lD_{X,-1}^n|_{\substack{X=g \\ Z= 3g-k-2}}p_k(X,Z)\binom{3g-3+l-k-Z}{l}\binom{g-X}{n} = p_k(X,Z),$$ which is true since $p_k$ is a polynomial of degree $k$ if we set $\text{deg}Z=1$ and $\text{deg}X=2$. Thus, we find another expression for the vector in the kernel of $M_k^T$, namely $$v_{k}^{a_1,n_1} = (-1)^{k-a_1-n_1}D_Z^{k-a_1-n_1}D_{X,-1}^n|_{\substack{X=g \\ Z= 3g-k-2}}p_k(X,Z), \ \ 0\le n_1\le a_1, \  a_1+n_1\le k.$$ The last formula tells us that the polynomials $p_k(X,Z)$ generate the solution to the equation in Theorem \ref{mainbeta} via subsequent applications of the discrete difference operators $D_Z$ and $D_{X,-1}$ and evaluations at the point $(X,Z) = (g,3g-k-2)$.
\end{remark}

\begin{remark}\label{heatkernel}
	We immediately verify that $$D_{Z,-1}^2K(X,Z;t) = -D_XK(X,Z;t),$$ thus for all $k$ we have \begin{equation}\label{discreteheat}
	D_{Z,-1}^2p_k(X,Z) = -D_Xp_k(X,Z),                                                                                        
	\end{equation}
	which has the shape of a heat equation in which both $X$ and $Z$ are discrete. Actually, it can be shown that for all $k\ge 1$, $p_k(X,Z)$ is the only polynomial solution to Equation \eqref{discreteheat} with initial condition $$p_k(0,Z) = \binom{Z}{k}.$$
\end{remark}

In the following Theorem, we see that the solution found in \ref{topperv} is the only one which solves Equation \eqref{eq} at the top defect.

\begin{theorem}\label{uniqueness}
	The kernel of $M_{k}^T$ is one-dimensional. Therefore, the lowest defect part $F_{k}$ of the solution to the equation in Theorem \ref{mainbeta} is unique.
\end{theorem}

\begin{proof}
	Thanks to Remark \ref{polyev}, we have to show that by letting $$\Gamma_k =\set{(X,Z)\in \mathbb{Z}^2 \ | \ X \ge 0, \ 2X \le Z \le X+k-1} $$ there exists only one (up to multiplication by a constant) polynomial of degree at most $k$ in $X$ and $Z$ (recall that $\text{deg} Z = 1$ and $\text{deg} X = 2$) which vanishes on $\Gamma_k$.
	
	Letting $p(X,Z)$ be such a polynomial, we can write it as $$p(X,Z) = \sum_{i=0}^{\lfloor k/2\rfloor }p_i(Z)\binom{X}{i}$$ where $\text{deg}( p_i) \le k-2i$. Since $p(0,Z) = p_0(Z)$ vanishes in $Z=0,1,\ldots, k-1$, we see that $p_0(Z) = \alpha_0\binom{Z}{k}$ for some constant $\alpha_0$. Then $$q_1(X,Z) = \frac{1}{X}(p(X,Z)-p(0,Z)) = \sum_{i=1}^{\lfloor k/2\rfloor}\frac{p_i(Z)}{i}\binom{X-1}{i-1}$$ and $q_1(1,Z) = p_1(Z)$ vanishes in the $k-2$ points $2,3,\ldots, k-1$, therefore $p_i(Z) = \alpha_i \binom{Z-2}{k-2}$. Continuing this way, we see that we can write $p(X,Z)$ as $$p(X,Z) = \sum_{i=0}^{\lfloor k/2\rfloor}\alpha_i\binom{Z-2i}{k-2i}\binom{X}{i}$$ for some constant $\alpha_i$. Now, for $ 1\le j \le \lfloor k/2\rfloor $, we have $2j\le k \le j+k-1$, therefore $$0=p(j,k) = \sum_{i=0}^{\lfloor k/2\rfloor}\alpha_i\binom{j}{i}, \ \text{for } 1 \le j \le \lfloor k/2\rfloor .$$ These are $\lfloor k/2\rfloor$ independent conditions on the $\alpha_i$'s, thus the space of polynomials of degree at most $k$ in $Z$ and $X$ which vanish on $\Gamma_k$ is at most one-dimensional. By the way, we know that the polynomial $p_k(X,Z) = \res_{t=0}\frac{\D t}{t^{k+1}}(1+t)^X(1+2t)^{Z-2X}$ satisfies this condition, therefore the space is exactly one-dimensional.
\end{proof}

\subsection{Factorization of $M_k$}

The matrix $M_k$ defined in \eqref{ratiomat} can be conveniently written as a product of two matrices $M_k = Q_kS_k$, where $S_k$ is a matrix with integer entries and $Q_k$ is triangular with $\pm 1$ on the diagonal. This factorization has some striking consequences in itself (see Proposition \ref{pervgrading}), and will be of key importance in Section \ref{betagamma}, where it will be used to give the determinantal criterion for P=W at top defect (see Corollary \ref{isalsoinvertible}).

Define the following matrix:
$$(Q_k)_{a,n}^{b,m}=(-1)^{k+b}e_{b+m-a-n}(3g-2-k,3g-1-k,\ldots, 3g-3-a-n)\cdot$$ \begin{equation}\label{qmatrix}\cdot e_{n-m}(g-n+1,g-n+2,\ldots,g)\end{equation} where $e_i(x_1,\ldots,x_j)$ is the $i$-th elementary symmetric polynomial in $x_1,\ldots,x_j$. Here we adopt the convention that $e_i(x_1,\ldots,x_j)=0$ if $i<0$ and ${e_i(x_1,\ldots,x_j)=1}$ if $0\le j<i$. The indexing of $Q_k$ satisfies $0\le n\le a\le k$, $a+n\le k$ and the same for $b$ and $m$. 

Let also $S_k$ be the matrix defined by \begin{equation}\label{fouriermat}(S_k)_{a_1,n_1}^{a_2,n_2}=n_2^{n_1} (a_2+n_2)^{k-n_1-a_1}\end{equation} where the rows are indexed by the pairs $(a_1,n_1)$ with $0\le n_1\le a_1\le k$, $n_1+a_1\le k$ and the columns by the pairs $(a_2,n_2)$ with $0\le n_2\le a_2\le k-1$. Here we adopt the convention that $0^0=1$ (hence the last row of $S_k$ is a row of $1$'s).

Then a direct computation shows that \begin{equation}\label{factorization}M_k=Q_kS_k\end{equation} and since $Q_k$ is invertible (being upper triangular with $\pm 1$ on the diagonal) we see that $\text{Ker}M_k=\text{Ker}S_k$. In particular, although the entries of $M_k$ are polynomials in $g$, its kernel is generated by vectors of $\mathbb{Z}^{k(k+1)/2}$, thus independent from $g$.

It is easy to write the inverse of $Q_k$ as 
$$(Q_k^{-1})_{a,n}^{b,m}=(-1)^{k+b}h_{b+m-a-n}(3g-2-k,3g-1-k,\ldots,3g-2-b-m) \cdot $$ \begin{equation}\label{inverseqmatrix}\cdot h_{n-m}(g,g-1,\ldots,g-m)\end{equation} 

where $h_r(x_1\ldots,x_n)$ is the \emph{complete symmetric function} of degree $r$, which is the sum of all monomials of total degree $r$ in the variables $x_1,\ldots,x_n$; by convention, $h_r = 0$ for $r<0$. The fact that this is indeed the inverse of $Q_k$ is a direct consequence of the classical identity $$\sum_{r=0}^n(-1)^re_rh_{n-r}=\delta_{0,n},$$ which is true for any number of variables.

With this we can show that, in the case of the classes $\beta^k$, the perverse filtration is actually a grading.

\begin{proposition}\label{pervgrading}
	For all $k\ge 1$, we have $\beta^k\notin P_{2k-1}(H^{4k}(\mathcal{M}))$. 
\end{proposition}

\begin{proof}
	Following Proposition \ref{pwlift} and Theorem \ref{enumerativepw}, we write the equation $$\int_Z\eta^{3g-4-2k}(\beta^k+\eta F) P =0 $$
	which must be in particular valid for all $P$ with $\text{def}(P)=2k$. We thus define $\widetilde{M}_k$ to be the matrix of the pairing $(F,P)\mapsto \int_Z\eta^{3g-4-2k}(\eta F)P$ with $\text{def}(F)=2k-2$, $\text{def}(P) = 2k$, and with an extra row corresponding to $\beta^k$. Then it is easy to see that if we define \begin{equation}(\widetilde{S}_k)_{a_1,n_1}^{a_2,n_2}=n_2^{n_1} (a_2+n_2)^{k-n_1-a_1}\end{equation} where the rows are indexed by the pairs $(a_1,n_1)$ with $0\le n_1\le a_1\le k$, $n_1+a_1\le k$ and the columns by the pairs $(a_2,n_2)$ with $0\le n_2\le a_2\le k$, then $$Q_k\widetilde{S}_k = \widetilde{M}_k$$ with notations as in \eqref{qmatrix}. Since $\widetilde{S}_k$ is obtained by $S_k$ by adding some columns, and since $\text{Ker}S_k^T$ is one-dimensional generated by the coefficients of $$p_k(X,Z) = \res_{t=0}\frac{\D t}{t^{k+1}}(1+t)^{Z-2X}(1+2t)^X,$$ then $\text{Ker}\widetilde{S}_k^T = (0)$ if and only if there exist values of $0\le n_2\le a_2 \le k$ such that ${p_k(n_2,a_2+n_2) \neq 0}$. Since $p_k(0,k) = 1$, we can conclude.
\end{proof}

\subsection{The classes $\beta^{k-h}\gamma^h$}\label{betagamma}

We now consider the general case of the extension problem, namely for $1\le k \le g-1$ and  $0\le h \le k$, \begin{equation}\label{general} \int_Z \eta^{3g-3-2k}(\beta^{k-h}(4\gamma)^h + \eta F)P = 0
\end{equation}
where we ask for the existence of some $F\in \C[\alpha,\beta,\gamma,\eta]$ such that (\ref{general}) is satisfied for all $P\in \C[\alpha,\beta,\gamma,\eta]$.

We shall still consider only the top-defect part of the pairing, namely we look for an $F$ with ${\text{def}(F) = 2k-2}$ such that Equation (\ref{general}) is satisfied for all $P$ with $\text{def}(P) = 2k-2$. Therefore, we see that $F$ is a sum of monomials of the form $$F_{k,h,a_1,n_1}\coloneqq \beta^{a_1-n_1}(4\gamma)^{n_1}\eta^{k-1-a_1}\alpha^{k+h-a_1-n_1}$$ for $ 0\le n_1\le a_1\le k-1,$ and $ a_1+n_1\le k+h.$ Analogously, $P$ is a sum of monomials of the form $$P_{k,h,a_2,n_2}\coloneqq \beta^{a_2-n_2}(4\gamma)^{n_2}\eta^{k-1-a_2}\alpha^{3g-3-k-h-a_2-n_2},$$ for $  0\le n_2\le a_2\le k-1$ and $ a_2+n_2 \le 3g-3-h-k.$ Notice that if we consider the general case $g\ge k+1$, the condition on $a_2+n_2$ is not redundant. However, from $n_2\le a_2\le k-1$ we can deduce $a_2+n_2\le 2k-2$, so the condition on $a_2+n_2$ becomes redundant if $2k-2\le 3g-3-h-k$, that is \begin{equation}\label{redundancy} 3g\ge 3k+h+1.\end{equation} Fixing values of $0\le h\le k$, we call the range of values of $g$ which satisfy \eqref{redundancy} the \emph{redundancy range}. The "smallest" case in which $g\ge k+1$ falls out the redundancy range is $k=h=3$, $g=4$.

By Proposition \ref{pwlift}, the top-defect part of the statement of the Enumerative P=W Conjecture in the redundancy range is the following.

\begin{conjecture}\label{P=Wgamma}
	Let $k\ge 1$ and $0\le h \le k$ be integers. In the redundancy range, there exists a unique defect-homogeneous $F\in H^*(Z)$  with $\emph{def}(F) = 2k-2$ such that for every defect-homogeneous $P\in H^*(Z)$  with $\emph{def}(P) = 2k-2$, Equation (\ref{general}) is satisfied.
\end{conjecture}

The case $h=0$ was proved in last section. We will give an equivalent statement for Conjecture \ref{P=Wgamma} in terms of the non-vanishing of a particular determinant involving the polynomials $p_k(X,Z)$ defined in the previous section.

Let \begin{equation}\label{gammamatrix} (M_{k,h})_{a_1,n_1}^{a_2,n_2} =  \binom{3g-3-a_1-n_1-a_2-n_2}{k+h-a_1-n_1}\binom{g-n_2}{n_1}\end{equation} for $0\le n_1\le a_1 \le k-1$ or $(n_1,a_1) = (h,k)$, and $0\le n_2\le a_2\le k-1$.

In the redundancy range, the matrix of the pairing $$(F,P)\mapsto \int_Z\eta^{3g-3-2k}(\eta F)P$$ for defect-homogeneous classes $F$ and $P$ with $\text{def}(F) = \text{def}(P) = 2k-2$, is $M_{k,h}$  without the last row $(n_1,a_1) = (h,k)$. Such row corresponds to multiplication with $\beta^{k-h}(4\gamma)^h$, therefore Conjecture \ref{P=Wgamma} is equivalent to stating that in the redundancy range, $\text{Ker}(M_{k,h}^T)$ is one-dimensional, and the entries of its vectors corresponding to the class $\beta^{k-h}\gamma^h$ are nonzero.

Notice that $M_{k,h}$ is a submatrix of $M_{k+h}$ of Definition \ref{ratiomat}. The number of columns of $M_{k,h}$ is  $$c_{k,h} = \frac{k(k+1)}{2}.$$ while the number of rows is \begin{equation}\label{rownumber} r_{k,h}=c_{k,h}-\left\lfloor \frac{(k-h+1)(k-h-3)}{4}\right\rfloor=c_{k+h,0}+\frac{h(h+3)}{2}.\end{equation}
In particular, if $h = k-3$, $M_{k,h}$ is a square matrix, if $k\ge h\ge k-2$ then $r_{k,h} = c_{k,h}+1$ and if $h\le k-4$ the number of rows is strictly smaller than the number of columns.

\begin{definition}
	For $0\le h \le k$, let $S_{k,h}$ be the matrix defined as $$(S_{k,h})_{a_1,n_1}^{a_2,n_2} = (a_2+n_2)^{k+h-a_1-n_1}n_2^{n_1}$$ with $0\le n_2\le a_2\le k-1$ and $0\le n_1\le a_1 \le k+h$, $a_1+n_1 \le k+h$; here we are using the convention $0^n = 0$ for all $n$ except $0^0=1$.
	
\end{definition}

\begin{lemma}\label{kerneldimension}
	The dimension of $\emph{Ker}(S_{k,h}^T)$ is $(h+1)(h+2)/2$. A basis of the kernel is given by the coefficients of the polynomials $$Z^{j}p_{k+h-i}(X,Z), \ \ 0\le j\le i\le h$$ where $p_k$ is the polynomial defined in \eqref{polydef}.
\end{lemma}

\begin{proof}
	Let $v= (v_{a_1,n_1})_{a_1,n_1}$ be a vector in $\text{Ker}(S_{k,h}^T)$. This means that the polynomial  $$p_v(X,Z) = \sum_{a_1,n_1}v_{a_1,n_1}X^{n_1}Z^{k+h-a_1-n_1}$$ vanishes at integers $X$ and $Z$ with $X\ge 0$ and $2X\le Z \le X+k-1$. All polynomials of the form $$Z^{j}p_{k+h-i}(X,Z), \ \ 0\le j\le i\le h$$ where $p_k$ is the polynomial defined in \eqref{polydef}, satisfy this condition. The number of such polynomials is $(h+1)(h+2)/2$, so let us show they are linearly independent. Since the case $h=0$ is the content of Theorem \ref{uniqueness}, suppose $k\ge h\ge 1$.

	We will prove the equivalent statement that the polynomials $$p_{i,j}(X,Z)\coloneqq \binom{Z-h-k+j}{j}p_{k+h-i}(X,Z), \ \ 0\le j\le i\le h$$ are linearly independent. To show this, simply form the matrix $$(T_h)_{0\le j \le i \le h}^{0\le a \le b \le h} = p_{i,j}(h-b,h+k-b+a) $$ and notice that, up to rearranging rows and columns, $T_h$ is a triangular matrix with powers of two as diagonal entries, thus it has nonzero determinant.
	We deduce that $$\text{dim}\text{Ker}S_{k,h}^T\ge (h+1)(h+2)/2.$$
	
	Conversely, assume $p(X,Z)$ is a polynomial of the form $$p(X,Z) = \sum_{\substack{0\le n \le a \\ n+a\le k+h}}\lambda_{n,a}X^nZ^{k+h-a-n}$$ which vanishes on $$\Gamma_k = \set{(X,Z)\in \Z^2 \ | \ 0\le X\le k-1, 2X\le Z \le X+k-1}.$$ If we define $\text{deg}Z=1$ and $\text{deg}X=2$, then $p(X,Z)$ is a polynomial of degree $k+h$. By changing basis, we can write it in the following way $$p(X,Z) = \sum_{i=0}^{\lfloor (k+h)/2\rfloor}p_i(Z)\binom{X}{i} $$ where $\text{deg}(p_i) = k+h-2i$. We see that $p_0(Z) = p(0,Z)$ vanishes on $Z=0,\ldots, k-1$, thus we must have $$p_0(Z) = \sum_{j=0}^h \alpha_{0,j}\binom{Z}{k+j} $$ for some constants $\alpha_{0,j}$. Then we consider $$q_1(X,Z) \coloneqq  \frac{1}{X}(p(X,Z)-p(0,Z)) = \sum_{i=1}^{\lfloor (k+h)/2\rfloor }\frac{p_i(Z)}{i}\binom{X-1}{i-1} $$ and we see that $p_1(Z) = q_1(0,Z)$ vanishes on $Z = 2, \ldots, k-1$, so that $$p_1(Z) = \sum_{j=0}^h \alpha_{1,j}\binom{Z-2}{k-2+j}.$$ We can continue this way up to $i=[k/2]$, deducing that $$p_i(Z) = \sum_{j=0}^h\alpha_{i,j}\binom{Z-2i}{k-2i+j}, \ \text{ for } i \le \lfloor k/2\rfloor .$$
	
	For $\lfloor k/2\rfloor < i \le \lfloor (k+h)/2\rfloor$, we do not have any conditions on the vanishing of the polynomials $p_i(Z)$, but we can write anyway $$p_i(Z) = \sum_{j= 2i-k }^{h}\alpha_{i,j}\binom{Z-2i}{k-2i+j}, \ \text{ for } \lfloor k/2 \rfloor \le i \le \lfloor (k+h)/2 \rfloor $$
	which is a general polynomial in $Z$ of degree $k+h-2i$. Putting everything together, we can write $$p(X,Z)=\sum_{i=0}^{\lfloor (k+h)/2 \rfloor}\sum_{j=\text{max}(0,2i-k)}^{h}\alpha_{i,j}\binom{Z-2i}{k-2i+j}\binom{X}{i}$$ and $ p(X,Z)$ vanishes on $0\le X \le \lfloor k/2\rfloor$, $2X \le Z \le k-1$. The vector space of such polynomials has dimension $$d_{k,h} = (\lfloor k/2 \rfloor + 1)(h+1) + \sum_{i=\lfloor k/2\rfloor + 1}^{\lfloor (k+h)/2\rfloor }(h+k-2i+1).$$ 
	
	We now impose conditions for $p(X,Z)$ to vanish on other points of $\Gamma_k$. Similarly to the proof of Theorem \ref{uniqueness}, setting $p(l,k) = 0$ for $1\le l \le \lfloor k/2 \rfloor$, we obtain $$\sum_{i=0}^{l} \binom{l}{i}\alpha_{i,0} = p(l,k) = 0 \ \text{ for } 1\le l \le \lfloor k/2 \rfloor,$$ which are $\lfloor k/2 \rfloor$ independent conditions on the $\alpha_{i,0}$ for $1 \le i \le \lfloor k/2 \rfloor $. Analogously, we have $p(l,k+1)=0 $ for $2\le l \le \lfloor (k+1)/2\rfloor$, which gives $$\sum_{i=0}^l\binom{l}{i}\left((k+1)\alpha_{i,0} + \alpha_{i,1} \right) = 0, \ \text{ for } 2 \le l \le \lfloor (k+1)/2 \rfloor .$$ With this we find $\lfloor (k+1)/2\rfloor-1$ independent conditions  on the $\alpha_{i,1}$'s for $2\le i \le \lfloor (k+1)/2 \rfloor$, once the $\alpha_{i,0}$ are chosen. Continuing this way, we obtain $\lfloor (k+j)/2 \rfloor - j$ independent linear conditions on the $\alpha_{i,j}$'s for all $0\le j \le h$. We conclude that $$\text{dim}\text{Ker}S_{k,h}^T\le d_{k,h} - \sum_{j=0}^h(\lfloor (k+j)/2 \rfloor - j) = (h+1)(h+2)/2 $$ where the last equality is tedious but straightforward to verify, thus completing the proof. \end{proof}

\begin{definition}
	Let $\{ v_1,\ldots, v_{(h+1)(h+2)/2}\}$ be independent vectors of $\text{Ker}(S_{k,h}^T)$. We define $\widetilde{Q}_{k,h}$ as the matrix obtained by replacing the last $(h+1)(h+2)/2$ rows of the matrix $Q_{k+h}(g,3g-3)$ defined in (\ref{qmatrix}) by the vectors $v_i$. 
\end{definition}

With the same method used to obtain \eqref{factorization}, we can show that $\widetilde{Q}_{k,h}$ satisfies  \begin{equation}\label{tildeq}\widetilde{Q}_{k,h}(g,3g-3)S_{k,h} = \widetilde{M}_{k,h}\end{equation} where $\widetilde{M}_{k,h}$ is obtained by $M_{k,h}$ by adding $h(h+3)/2$ zero rows, and by replacing the row corresponding to $\beta^{k-h}(4\gamma)^h$ with a zero row. Analogously, if we define $\widehat{Q}_{k,h}$ to be $Q_{k+h}$ with the last rows except the one corresponding to $\beta^{k-h}\gamma^h$ replaced with $h(h+3)/2$ independent vectors of $\text{Ker}S_{k,h}^T$, we get $$\widehat{Q}_{k,h}(g,3g-3)S_{k,h} = \widehat{M}_{k,h} $$ where $\widehat{M}_{k,h}$ is obtained by $M_{k,h}$ by adding $h(h+3)/2$ zero rows.

Thus, information on $\widetilde{Q}_{k,h}$ and $\widehat{Q}_{k,h}$ would lead to the solution of the matrix problem. A particularly good result would follow if $\widetilde{Q}_{k,h}$ were invertible. 

\begin{lemma}\label{lemmadet}
	Let $p_k(X,Z)$ be the polynomials defined in \eqref{polydef}.
	\begin{enumerate} \item[1.] Let $$ B_{k,h}(X,Z) = (Z^jp_{k+h-i}(X-n,Z-m))_{0\le j \le i \le h}^{0\le n\le m\le h}.$$ 
		Then we have $\emph{det}\widetilde{Q}_{k,h}(g,3g-3) = c_{k,h}\emph{det}B_{k,h}(g,3g-2-k)$ for some nonzero constants $c_{k,h}$.

		\item[2.]  $\emph{det}\widehat{Q}_{k,h}(g,3g-3)$ is the determinant of a first minor of $B(g,3g-2-k)$.
		
		\item[3.] Define \begin{equation}\label{ourdet}W_{k,h}(X,Z) \coloneqq  \emph{det}(p_k(X-h+i,Z+j))_{0\le i \le h}^{0\le j \le h}.\end{equation} Then, up to a nonzero constant, $$\emph{det}(\widetilde{Q}_{k,h}(g,3g-3))=\prod_{i=0}^hW_{k,i}(g,3g-k-i-2).$$
		
	\end{enumerate}
\end{lemma}

\begin{proof}
	Let us choose the basis of $\text{Ker}(S_{k,h}^T)$ given by the polynomials $Z^jp_{k+h-i}(X,Z)$ with $0\le j  \le i \le h$ and let us construct $Q_{k,h}$ accordingly.
	
	We consider the matrix $C_{k,h}=\widetilde{Q}_{k,h}(Q_{k+h})^{-1}$, where $Q_{k+h}$ is the matrix defined in \eqref{qmatrix}. This is a block upper-triangular matrix with a square block of size $(h+1)(h+2)/2$ at the bottom-right corner and with all other blocks of size $1$, each containing a $1$; thus our determinant is equal to the determinant of the bottom-right block (up to a sign: it is easy to see that this sign is $+$, since the bottom-right minor of size $(h+1)(h+2)/2$ of $Q_{k+h}$ has determinant $1$).
	
	To compute it we use (\ref{inverseqmatrix}) along with the identity $$\sum_{i=0}^a(-1)^i\frac{a!}{(a-i)!}h_{b-i}(x,\ldots, x-i) = (x-a)^b$$ which is valid for all integers $a$ and $b$: with this we can take appropriate linear combinations of the last $(h+1)(h+2)/2$ columns of $(Q_{k+h})^{-1}$ to obtain a matrix whose entry $((a,n),(b,m))$ is $(y-b-m+1)^{k+h-a-n}(x-m)^n$ in the last $(h+1)(h+2)/2$ columns. This amounts to performing column operations on $\widetilde{Q}_{k,h}(Q_{k+h})^{-1}$, so this procedure does not change its determinant up to multiplying by nonzero constants. 
	
	Recalling that the last $(h+2)(h+1)/2$ rows of $\widetilde{Q}_{k,h}$ consist of the coefficients of $Z^jp_{k+h-i}(X,Z)$, the determinant is then equal to the one in the statement and this proves Point 1. Point 2 is shown similarly when the row which is eliminated corresponds to the polynomial among $(Z^jp_{k+h-i})_{i,j}$ which is not considered, and the column is the index of $Q_{k+h}$ corresponding to the class $\beta^{k-h}(4\gamma)^h$. To show Point 3, use repeatedly the identity $p_k(X,Z+1)-p_k(X,Z) = p_{k-1}(X,Z)$ in the matrix $B_{k,h}$.
\end{proof}

Point 2. of the previous Lemma immediately yields the following.

\begin{corollary}\label{isalsoinvertible}
	If $\widetilde{Q}_{k,h}(g,3g-3)$ is invertible, then there exists a vector in $\emph{Ker}S^T_{k,h}$ which can be discarded to give an invertible $\widehat{Q}_{k,h}(g,3g-3)$.
\end{corollary}

\begin{theorem}\label{ifinvertible}
	Assume $\widetilde{Q}_{k,h}(g,3g-3)$ is invertible. Then there exists a unique solution to Equation \eqref{general}.
\end{theorem}

\begin{proof}
	If $\widetilde{Q}_{k,h}$ is invertible, then by Corollary \ref{isalsoinvertible} we can choose a $\widehat{Q}_{k,h}$ that is also invertible. This proves that $$\text{rk}\widetilde{M}_{k,h} = \text{rk}\widehat{M}_{k,h}$$ since both are equal to $\text{rk}S_{k,h}$. Now since $\widehat{M}_{k,h}$ is obtained by $\widetilde{M}_{k,h}$ by replacing a zero row with the row corresponding to $\beta^{k-h}\gamma^h$, we deduce that such row must be a linear combination of the others, thus giving a solution to \eqref{general}.
	
	Now since $\text{Ker}S_{k,h}^T$ has dimension $(h+2)(h+1)/2$, and since $\widetilde{M}_{k,h}$ has exactly $(h+2)(h+1)/2$ zero rows, it follows that $\text{Ker}M_{k,h}^T$ must be one-dimensional, therefore the solution to \eqref{general} is unique.
\end{proof}

Extensive numerical computations has brought us to state the following.

\begin{conjecture}\label{whatwewant}
	In the redundancy range, $W_{k,h}(g,3g-k-h-2)>0.$
\end{conjecture}
Conjecture \ref{whatwewant} would imply the existence and uniqueness of the lowest defect part of the solution to Equation (\ref{general}) in the redundancy range.
Here we provide a proof for $h=1$.

\begin{proposition}\label{h=1}
	If $g\ge k+1$, then $W_{k,1}(g,3g-k-3)>0.$
\end{proposition}

\begin{proof}
	With a simple change of variables, we can rewrite the polynomials $p_k$ as $$p_k(X,Z) = \res_{t=0}\frac{\D t}{t^{k+1}}(1+t)^X(1-t)^{X-Z+k-1}$$ therefore, we can rewrite the determinants $W_{k,1}(X,Z)$ as
	$$\res_{t_0,t_1=0}\frac{\D t_0\D t_1}{t_0^{k+2}t_1^{k+1}}((1+t_0)(1+t_1))^{X-1}(1-t_0)^{X-Z+k-1}(1-t_1)^{X-Z+k-2}(t_0^2-t_1^2).$$ We decompose it as $$W_{k,1} = W_{k,1}^0-W_{k,1}^1$$ accordingly to the terms of the Vandermonde factor $t_0^2-t_1^2$. Then we have the following
	\begin{lemma}
		For all $k\ge 1$, we have $$(k-1)W_{k,1}^0(X,Z)-(k+1)W_{k,1}^1(X,Z) = 2(X-1)W_{k-1,1}(X-1,Z-2) $$                                                                                                                                                                                                                                                                                     
	\end{lemma}
	
	\begin{proof}[Proof of the lemma]
		We immediately see that $$W_{k,1}^0(X,Z) = p_k(X-1,Z)p_{k-1}(X-1,Z-2), $$  $$W_{k,1}^1(X,Z) = p_{k+1}(X-1,Z)p_{k-2}(X-1,Z-2).$$ Moreover we have the easily proven formulas $$ (k+1)p_{k+1}(X,Z) = X p_k(X-1,Z-2)+(Z-X-k)p_k(X,Z),$$ $$p_k(X-1,Z) = p_{k-1}(X-2,Z-2) + p_k(X-2,Z-1), $$ $$ p_k(X,Z) = p_k(X,Z-1)+p_{k-1}(X,Z-1). $$ 
		Applying in order the first, second and third formula to the statement, we get the result.
	\end{proof}

	Now we can prove the Proposition \ref{h=1} by induction on $k$. Assume that we found the domain where $(X-1)W_{k-1,1}(X-1,Z-2)\ge0$, then $$\frac{W_{k,1}^1(X,Z)}{W_{k,1}^0(X,Z)} \le \frac{k-1}{k+1}<1$$ if $W_{k,1}^0(X,Z) >0$, thus yielding the domain for $W_{k,1}$ by the Lemma.

	From the expression of $W_{k,1}^0$ given in the proof of the Lemma, we easily find $$W_{k,1}^0(X,Z) > 0 \text{ for } \begin{cases}
	X\ge k, \\
	Z \ge X+k-1. \end{cases}$$
	
	Now since $W_{1,1}(X,Z) = Z$, we find $$W_{k,1}(X,Z)> 0 \text{ for } \begin{cases} X\ge k, \\
	Z\ge X+k-1.\end{cases}$$ Finally, if $g\ge k+1$, we have $3g-k-3 \ge g + k -1$ and the proof is complete.
\end{proof}

\begin{remark}\label{whatwehave}
We also managed to prove that $W_{k,h}(g,3g-k-h-2)>0$ for $g\ge k+h+2$. The proof will be provided in a forthcoming work.
\end{remark}

\begin{remark}
 Notice that, in any case, $W_{k,h}(g,3g-k-h-2)$ is a nonzero polynomial in $g$ with positive leading term. Therefore, existence and uniqueness of the solution to Equation \eqref{general} is assured for $g$ big enough (depending on $k$ and $h$).
\end{remark}

\begin{remark}
	If $g\ge k+1$ is outside the redundancy range, then the matrix of the pairing $(F,P)\mapsto \int_Zu^{3g-3-2k}(\eta F)P$ can still be defined with the same formula (and the same rows and columns range) as in Definition \ref{gammamatrix}: it is indeed easily shown that if $3k+3\le 3g\le 3k+h$, then the extra columns in $M_{k,h}$ are automatically zero. Thus, the solution to Equation (\ref{general}), even outside the redundancy range, is given by an element of $\text{Ker}M_{k,h}^T$ whose entry corresponding to $\beta^{k-h}(4\gamma)^h$ is nonzero.
	
	However, in this range, computer calculations have shown that, although a solution to Equation \eqref{general} still exists, we have $\text{det}\widetilde{Q}_{k,h}=0$ and the solution is never unique at the level of polynomials in $\alpha$, $\beta$, $\gamma$ and $\eta$. By considering the difference of two such solutions, we have noticed that they come from the relations in $H^*(\M)$ described in \cite{hausel-thaddeus2}. Therefore, we conjecture that the solution to Equation \eqref{general} is in any case unique in cohomology.
\end{remark}

\end{document}